\documentclass[a4paper,10pt,twoside]{article}

\usepackage{proof}

\input{dst.sty}

\begin{document}

\title{A functional interpretation for nonstandard arithmetic}
\author{Benno van den Berg\footnote{Mathematisch Instituut, Universiteit Utrecht,
PO. Box 80010, 3508 TA Utrecht. Email address: B.vandenBerg1@uu.nl. Supported by the Netherlands Organisation for Scientific Research (NWO).
}, Eyvind Briseid\footnote{
Department of Mathematics,
The University of Oslo,
Postboks 1053, Blindern,
0316 Oslo. Email address: eyvindbriseid@gmail.com. Supported by the Research Council of Norway (Project 204762/V30).}, and Pavol Safarik\footnote{Fachbereich Mathematik, Technische Universit\"{a}t Darmstadt,
Schlo{\ss}gartenstra{\ss}e 7, 64289 Darmstadt. Email address:  pavol.safarik@googlemail.com. Supported by the German Science Foundation (DFG project KO 1737/5-1).}}
\date{July 12, 2012}
\maketitle

\begin{abstract}
\noindent
We introduce constructive and classical systems for nonstandard arithmetic and show how variants of the functional interpretations due to G\"odel and Shoenfield can be used to rewrite proofs performed in these systems into standard ones. These functional interpretations show in particular that our nonstandard systems are conservative extensions of $\eha$ and $\epa$, strengthening earlier results by Moer\-dijk and Palmgren, and Avigad and Helzner. We will also indicate how our rewriting algorithm can be used for term extraction purposes. To conclude the paper, we will point out some open problems and directions for future research, including some initial results on saturation principles.
\end{abstract}

\tableofcontents

\section{Introduction}

In this paper we present functional interpretations for both constructive and classical systems of nonstandard arithmetic. The interpretations have two aspects: they show that the nonstandard systems are conservative over ordinary (standard) ones and they show how terms can be extracted from nonstandard proofs.

After inventing nonstandard analysis and showing that it was a versatile idea leading to nonstandard proofs in various areas of mathematics, Robinson suggested one could look at nonstandard arguments ``syntactically'' (proof-theoretically). From this point of view, one would see nonstandard analysis as ``introduc[ing] \emph{new deductive procedures} rather than new mathematical entities'' \cite{robinson96}. Apparently he did not think that adding these new deductive procedures to standard systems would make them more powerful or stronger in a proof-theoretic sense. In fact, in \cite{robinson73} he formulates as a general problem ``\emph{to devise a purely syntactic transformation which correlates standard and nonstandard proofs of the same theorems in a large area, e.g., complex function theory}''. In \cite{kreisel67} he is reported as asking a more specific question, whether a certain system for nonstandard arithmetic is conservative over $\PA$. These ideas of Robinson have borne fruit in the work of Kreisel \cite{kreisel67, kreisel69}, Friedman (unpublished), Nelson \cite{nelson77, nelson88}, Moerdijk, Palmgren \cite{moerdijk95, moerdijkpalmgren97, palmgren00}, Avigad \cite{avigadhelzner02, avigad05} and others, who proved for various systems for nonstandard analysis that they are conservative extensions of standard systems. Often their arguments are effective in that one could extract algorithms from their proofs which convert nonstandard arguments into standard ones. 

As an example of this, let us have a short look at the work of Nelson, also because it is a major source of inspiration for this article. The idea of Nelson was to add a new unary predicate symbol $\st$ to $\ZFC$ for ``being standard''. In addition, he added three new axioms to $\ZFC$ governing the use of this new unary predicate, called Idealization, Standardization and Transfer. The resulting system he called $\IST$, which stands for Internal Set Theory. The main logical result about $\IST$ is that it is a conservative extension of $\ZFC$, so any theorem provable in $\IST$ which does not involve the $\st$-predicate is provable in $\ZFC$ as well. Hence such theorems are genuine mathematical results, acceptable from the generally shared foundational standpoint of $\ZFC$.

The conservativity of $\IST$ over $\ZFC$ was proved twice. In the original paper where he introduces Internal Set Theory \cite{nelson77} (recently reprinted with a foreword by G.~F. Lawler in Volume 48, Number 4 of the \emph{Bulletin of the American Mathematical Society} in recognition of its status as a classic), Nelson gives a model-theoretic argument which he attributes to Powell. In a later publication \cite{nelson88}, he proves the same result syntactically by providing a ``reduction algorithm'' (a rewriting algorithm) for converting proofs performed in ${\IST}$ to ordinary ${\ZFC}$-proofs. There is a remarkable similarity between his reduction algorithm and the Shoenfield interpretation \cite{shoenfield01}; this observation was the starting point for this paper.

We will work with systems in higher types, such as $\ha$ and $\pa$, rather than set theory, because, as we mentioned before, we will not just be interested in establishing conservation results, but also in extracting terms from nonstandard proofs and ``proof mining''. Proof mining is an area of applied logic in which one uses proof-theoretic techniques to extract quantitive information (such as bounds on the growth rate of certain functions) from proofs in ordinary mathematics. In addition, such techniques can reveal certain uniformities leading to new qualitative results as well. Functional interpretations are one of the main tools in proof mining (for an introduction to this part of applied proof theory, see \cite{Kohlenbach08}). To extract interesting bounds, however, it is important that the mathematical arguments one analyses can be performed in sufficiently weak systems: therefore one considers systems such as $\ha$ or $\pa$, or fragments thereof, rather than $\ZFC$. The reason for considering systems in higher types (rather than $\PA$, for instance) is not just because they are more expressive, but also because higher types are precisely what makes functional interpretations work.

Although establishing conservation and term extraction results for systems of nonstandard arithmetic is what this paper is about, there is another way of looking at the results of this paper, which has more to do with the ideas of Lifschitz on calculable numbers \cite{lifschitz85}, of Berger on uniform Heyting arithmetic \cite{berger05} and of Hernest on the light Dialectica interpretation \cite{hernest05}, than with nonstandard arithmetic. Their idea was to have two kinds of quantifiers, one with computational content and one without. On the realizability interpretation of Lifschitz, the computationally meaningful  quantifiers are interpreted in the usual way (which, in the case of the existential quantifier, means that a realizer needs to exhibit a witness, while a realizer for a universal statement $\forall n \, \varphi(n)$ is a program which computes a realizer of $\varphi(n)$ from the value $n$). The computationally empty quantifiers, on the other hand, are to be interpreted uniformly (which means that it need not exhibit a witness in the existential case: a witness simply has to exist; while in the case of the universal quantifier it has to be a realizer which, uniformly, realizes $\varphi(n)$ for all $n$). 

A new unary predicate $\st$ introduces two types of quantifiers as well: the \emph{internal} quantifiers $\forall x$ and $\exists x$, as well as the \emph{external} quantifiers $\forallst x$ and $\existsst x$ (which can be seen as abbreviations of $\forall x \, ( \, \st(x) \to \ldots)$ and $\exists x \, ( \, \st(x) \land \ldots)$ respectively). Our initial idea was to interpret the former uniformly, while interpreting the latter in the usual way, in complete analogy with the ideas of Lifschitz. But this led to several, to us, undesirable effects (in particular, one could not realize the statement that the standard natural numbers are closed downwards). Our solution was to weaken the computational meaning of the external quantifiers: in particular, to realize $\existsst n \, \varphi(n)$ it suffices to exhibit a finite lists of candidates $n_1, \ldots, n_k$ such that at least one of the statements $\varphi(n_i)$ is realized. Because of the analogy with Herbrand disjunctions from proof theory, we have dubbed this type of realizability ``Herbrand realizability''.

To make Herbrand realizability work in a higher type setting, it is convenient to work in an extension of $\ha$ with types for finite sequences. More precisely, we will assume that there is a type $\sigma^*$ for sequences of objects of type $\sigma$. The type $\sigma^*$ carries the structure of a preorder (with $x \preceq y$ if every element in the list coded by $x$ also occurs in the list coded by $y$) and one would naturally expect realizers to be closed upwards with respect to this preorder. To make this work nicely, it will be useful to introduce a new kind of application (of functions to arguments) which is monotone in the first component. This can be done and with this additional ingredient Herbrand realizability can be defined. We will do this in Section 4.

It has to be admitted that the connection of Herbrand realizability to nonstandard arithmetic is rather tangential and, as a matter of fact, we will not be very interested in Herbrand realizability \emph{per se}. This will change radically if we turn to the functional interpretation introduced in Section 5, which bears the same relation to Herbrand realizability as the usual Dialectica interpretation does to modified realizability. It turns out that if one defines a Dialectica-type functional interpretation using the new application, with implication interpreted \emph{\`a la} Diller--Nahm \cite{dillernahm74}, basing it on some of the characteristic principles of Herbrand realizability, and the idea of having Herbrand disjunctions realize existential statements, this interpretation will, almost by magic, interpret and eliminate principles recognizable from nonstandard analysis. Our main reason for including Herbrand realizability is to have an easy point of access to and to provide some intuition for this functional interpretation.

So far the techniques we mentioned work only for constructive systems. But by combining the functional interpretation we mentioned above with negative translation, we are able to define a Shoenfield-type functional interpretation for classical nonstandard systems as well. In this way we also obtain conservation and term extraction results for classical systems. We will work out the technical details in Sections 6 and 7 below.

The resulting functional interpretations have some striking similarities with the bounded functional interpretations introduced by Ferreira and Oliva in \cite{ferreiraoliva05} and \cite{ferreira09} (see also \cite{gaspar09}). In the same way, Herbrand realizability seems related to the bounded modified realizability interpretation due to Ferreira and Nunes (for which, see \cite{ferreiranunes06}). We will briefly comment on this in Section 5 below.

The contents of this paper are therefore as follows. In Section 2 we will introduce an intuitionistic base system for our investigations into constructive nonstandard arithmetic. In Section 3, we will discuss some principles from nonstandard analysis and their relations. This will give one some ideas of how interpretations of nonstandard systems have to look like and will provide us with some ``benchmarks'' with which one can measure the success of an interpretation. In Section 4 we will introduce Herbrand realizability and discuss its merits as an interpretation of nonstandard arithmetic. Our nonstandard functional interpretation will be introduced in Section 5 and we will use it to prove several conservation and term extraction results in an intuitionistic context. In Section 6 we introduce a classical base system for nonstandard arithmetic and discuss two variants of the negative translation. These we will use in Section 7 to obtain a Shoenfield-type functional interpretation and derive conservation and term extraction results in a classical context. Finally, Section 8 discusses work in progress on saturation principles and other plans for future work.

We would like to thank Paulo Oliva, the participants in the Spring 2010 proof theory seminar at the Technische Universit\"at Darmstadt, especially Ulrich Kohlenbach, Jaime Gaspar, and Alexander Kreuzer, and the referee for helpful comments.
\section{Formalities}

In this section, we introduce our base systems for investigating nonstandard arithmetic. 

\subsection{The system $\ehastar$}

In this paper, $\ehastar$ will be the extension of the system called $\ehazero$ in \cite{Troelstra73} and $\ehaarrow$ in \cite{troelstravandalen88b} with types for finite sequences. More precisely, the collection of types $\Tpstar$ will be smallest set closed under the following rules:
\begin{enumerate}
\item[(i)] $0 \in \Tpstar$;
\item[(ii)] $\sigma, \tau \in \Tpstar \Rightarrow (\sigma \to \tau) \in \Tpstar$;
\item[(iii)] $\sigma \in \Tpstar \Rightarrow \sigma^* \in \Tpstar$.
\end{enumerate}

Because we have not included product types in $\Tpstar$, we will often be handling tuples of types or terms. We will always refer to such lists of types and terms as tuples and never as sequences, so as not to confuse them with terms of sequence type (i.e., of type $\sigma^*$ for some $\sigma \in \Tpstar$). In dealing with tuples, we will follow the notation and conventions of \cite{Troelstra73} and \cite{Kohlenbach08}. In particular, if $\tup x = x_0, \ldots, x_{m-1}$ and $\tup y = y_0, \ldots, y_{n-1}$, then
\begin{enumerate}
\item $[]$ stands for the empty tuple, while $\tup x, \tup y$ stands for $x_0, \ldots, x_{m-1}, y_0, \ldots, y_{n-1}$;
\item $x_i\tup y$ stands for $( \ldots ((x_iy_0)y_1) \ldots)y_{n-1}$, while $\tup x \tup y$ stands for $x_0\tup y, \ldots, x_{m-1}\tup y$ (and never for $x_0, \ldots, x_{m-1}\tup y$);
\item $\lambda \tup x. \tup y$ stands for $\lambda \tup x.y_0, \ldots, \lambda \tup x. y_{n-1}$;
\item and, finally, if $\underline{x}=x^{\sigma_0}_0,\ldots,x^{\sigma_{m-1}}_{m-1}$ and $\underline{y}=y^{\sigma_0}_0,\ldots,y^{\sigma_{m-1}}_{m-1}$ are tuples having the same length and types, we will write $\tup{x} =_{\tup{\sigma}} \tup{y}$
for
\[
        \bigwedge_{j=0}^{m-1} \, x_j =_{\sigma_j} y_j.
\]
\end{enumerate}

Because we have included sequence types, we will have to enrich the term language (G\"odel's $\T$); it now also includes a constant $\et_\sigma$ of type $\sigma^*$ and an operation $c$ of type $\sigma \to (\sigma^* \to \sigma^*)$ (for the empty sequence and the operation of prepending an element to a sequence, respectively), as well as a list recursor $\tup L_{\sigma, \tup \rho}$ satisfying the following axioms:
\begin{eqnarray*}
\tup L_{\sigma, \tup \rho}\,  \et_\sigma \tup y \tup z & =_{\tup \rho} & \tup y, \\
\tup L_{\sigma, \tup \rho} \, c(a, x) \tup y \tup z & =_{\tup \rho} & \tup z({\tup L}_{\sigma, \tup \rho}\, x\tup y \tup z) a,
\end{eqnarray*}
where $\tup \rho = \rho_1, \ldots, \rho_k$ is a $k$-tuple of types, $\tup y = y_1, \ldots, y_k$ is a $k$-tuple of terms with $y_i$ of type $\rho_i$ and $\tup z = z_1, \ldots, z_k$ is a $k$-tuple of terms with $z_i$ of type $\rho_1 \to \ldots \to \rho_k \to \sigma \to \rho_i$
(compare \cite[p. 456]{troelstravandalen88b} or \cite[p. 48]{Kohlenbach08}). In addition, we have the recursors and combinators for all the new types in G\"odel's $\T$, satisfying the usual equations.  The resulting extension we will denote by $\Tstar$.

We will have a primitive notion of equality at every type and equality axioms expressing that equality is a congruence (as in \cite[p. 448-9]{troelstravandalen88b}). Since decidability of quantifier-free formulas is not essential for this paper, this choice will not create any difficulties. In addition, we assume the axiom of extensionality for functions:
\begin{displaymath}
\begin{array}{l}
f =_{\sigma \to \tau} g \leftrightarrow \forall x^\sigma \, fx =_{\tau} gx.
\end{array}
\end{displaymath}

Of course, the underlying logic of $\ehastar$ is constructive; we will assume it is axiomatized as in \cite[p. 42]{Kohlenbach08}. We also have all the usual axioms of $\eha$, as in \cite[p.~48-9]{Kohlenbach08}, for example, where it is to be understood that the induction axiom applies to all formulas in the language (i.e., also those containing variables of sequence type and the new terms that belong to $\Tstar$). Finally, we add the following sequence axiom:
\[ \SA: \quad \forall y^{\sigma^*} \, ( \, y = \et_\sigma \lor \exists a^\sigma, x^{\sigma^*} \, y = c(a,x) \, ). \]
In normal $\eha$, as in \cite{Kohlenbach08} or \cite{troelstravandalen88b}, for example, one can also talk about sequences, but these have to be coded up (see \cite[p. 59]{Kohlenbach08}). As a result, $\ehastar$ is a definitional extension of, and hence conservative over, $\eha$ as defined in \cite{Kohlenbach08} or \cite{troelstravandalen88b}.

\subsection{The system $\ehaststar$}

\begin{dfn} The language of the system $\ehaststar$ is obtained by extending that of $\ehastar$ with unary predicates $\st^\sigma$ as well as two new quantifiers $\forallst x^\sigma$ and $\existsst x^\sigma$ for every type $\sigma \in \Tpstar$. Formulas in the language of $\ehastar$ (i.e., those that do not contain the new predicate $\st_\sigma$ or the two new quantifiers $\forallst x^\sigma$ and $\existsst x^\sigma$) will be called \emph{internal}. Formulas which are not internal will be called \emph{external}.
\end{dfn}

We will adopt the following
\begin{quote}
{\large \sc{Important convention:}} We follow Nelson \cite{nelson88} in using small Greek letters to denote internal formulas and capital Greek letters to denote formulas which can be external.
\end{quote}

\begin{dfn}[$\ehaststar$] The system $\ehaststar$ is obtained by adding to $\ehastar$ the axioms $\EQ, \Tst$ and  $\IA^{\st{}}$, where
\begin{itemize}
\item $\EQ$ stands for the defining axioms of the external quantifiers:
\begin{eqnarray*}
\forallst x \, \Phi(x) & \leftrightarrow &  \forall x \, (\, \st(x)\rightarrow\Phi(x) \, ),\\
\existsst x \, \Phi(x) & \leftrightarrow & \exists x \, (\, \st(x)\wedge\Phi(x) \, ),
\end{eqnarray*}
with $\Phi(x)$ an arbitrary formula, possibly with additional free variables.
\item $\Tst$ consists of:
\begin{enumerate}
\item the axioms $\st(x) \land x = y \to \st(y)$,
\item the axiom $\st(t)$ for each closed term $t$ in $\Tstar$,
\item the axioms $\st(f)\wedge\st(x)\rightarrow\st(fx)$.
\end{enumerate}
\item $ \IA^{\st{}}$ is the external induction axiom:
\[
\IA^{\st{}} \quad : \quad\big(\Phi(0)\wedge\forallst n^0 (\Phi(n)\rightarrow\Phi(n+1) )\big)\rightarrow\forallst n^0 \Phi(n),
\]
where $\Phi(n)$ is an arbitrary formula, possibly with additional free variables.
\end{itemize}
Here it is to be understood that in $\ehaststar$ the laws of intuitionistic logic apply to all formulas, while the induction axiom from $\ehastar$
\[ \quad\big(\varphi(0)\wedge\forall n^0 (\varphi(n)\rightarrow\varphi(n+1) )\big)\rightarrow\forall n^0 \varphi(n) \]
applies to internal formulas $\varphi$ only.
\end{dfn}

\begin{lemma} \label{congruence}
$\ehaststar \vdash \Phi(x) \land x= y\to \Phi(y)$ for every formula $\Phi$.
\label{le:extensionality}
\end{lemma}
\begin{proof}
By induction on the logical structure of $\Phi$.  Note that the $\st$-predicate is extensional by $\Tst$ and the case of the external quantifiers $\forallst$ and $\existsst$ can be reduced to that of the internal quantifiers $\forall$ and $\exists$ by using the $\EQ$-axiom.
\end{proof}

\begin{lemma}
$\ehaststar \vdash \st^0(x) \land y \leq x \to \st^0(y)$.
\end{lemma}
\begin{proof}
Apply external induction to the formula $\Phi(x) :\equiv \forall y \,( \, y \leq x \to \st(y))$.
\end{proof}

\begin{remark}
The previous lemma implies that, \emph{whenever $n$ is a standard natural number}, a bounded internal quantifier of the form $\existsst i \leq n$ can always be replaced by $\exists i \leq n$ and vice versa (the same applies to the universal quantifiers, of course). So we can regard such bounded quantifiers as internal or external, depending on what suits us best. Most often, however, it will be convenient to regard them as internal quantifiers.
\end{remark}

\begin{dfn}
For any formula $\Phi$ in the language of $\ehaststar$, we define its \emph{internalization} $\Phi^{\intern}$ to be the formula one obtains from $\Phi$ by replacing $\st(x)$ by $x = x$, and $\forallst x$ and $\existsst x$ by $\forall x$ and $\exists x$, respectively.
\end{dfn}

One of the reasons $\ehaststar$ is such a convenient system for our proof-theoretic investigations is because we have the following easy result:

\begin{prop} \label{conservativeint}
If a formula $\Phi$ is provable in $\ehaststar$, then its internalization $\Phi^{\intern}$ is provable in $\ehastar$. Hence $\ehaststar$ is a conservative extension of $\ehastar$ and $\eha$.
\end{prop}
\begin{proof}
Clear, because the internalizations of the axioms of $\ehaststar$ are provable in $\ehastar$.
\end{proof}

\subsection{Operations on finite sequences}

Using the list recursor $L_{\sigma, \tup \rho}$ one can define a length function $|\cdot|: \sigma^* \to 0$ satisfying
\begin{eqnarray*}
|\et_\sigma| & = & 0, \\
|c(a, x)|  & = &S|x|.
\end{eqnarray*}
Moreover, we can fix for every type $\sigma$ a term $\mathcal{O}_\sigma$ in $\Tstar$ of that type. One can then also define a projection function $\sigma^* \to (0 \to \sigma)$; we will write $(x)_i$ for the $i$th projection of $x$. It satisfies:
\begin{eqnarray*}
(\et_\sigma)_n & = & \mathcal{O}_\sigma, \\
(c(a, x))_0 & = & a, \\
(c(a, x))_{Sn} & = & (x)_n. 
\end{eqnarray*}
In addition, we will have an operation which given $x_0, \ldots, x_{n-1}$ of type $\sigma$ builds an object $x = \langle x_0, \ldots, x_{n-1} \rangle$ of type $\sigma^*$ for which we have $|x| = n$ and
\begin{displaymath}
\begin{array}{lcll}
(x)_i & = & x_i & \mbox{if } i < |x|, \\
(x)_i & = & \mathcal{O}_\sigma & \mbox{otherwise}.
\end{array}
\end{displaymath}
We will also need a concatenation operation $*_{\sigma^*}: \sigma^* \to (\sigma^* \to \sigma^*)$ defined by:
\begin{eqnarray*}
\et_\sigma *_{\sigma^*} y & = & y, \\
c(a,x) *_{\sigma^*} y & = & c(a, x*y).
\end{eqnarray*}
This we can use to define an $n$-fold concatenation: If $F: 0 \to \sigma^*$ and $n$ is of type 0, then we can set:
\begin{eqnarray*}
(F(0) *_{\sigma^*} \ldots *_{\sigma^*} F(n-1)) & = & \left\{ \begin{array}{ll}
\et_\sigma & \mbox{if } n = 0, \\
(F(0) *_{\sigma^*} \ldots *_{\sigma^*} F(n-2)) *_{\sigma^*} F(n-1) & \mbox{if } n > 0.
\end{array} \right.
\end{eqnarray*}
Note that $F(0) *_{\sigma^*} \ldots *_{\sigma^*} F(n-1) = F(0)$ if $n = 1$.

\begin{lemma} \label{presstandardness}
     \begin{enumerate}
       \item $\ehaststar \vdash \st(x^{\sigma^*}) \to \st (|x|),$
      \item $\ehaststar \vdash \st(x^{\sigma^*}) \to \st ((x)_{i}),$
\item $\ehaststar \vdash \st(x^{\sigma}_0) \land \ldots \land \st(x^{\sigma}_n) \to \st (\langle x^{\sigma}_0,\ldots,x^{\sigma}_n\rangle),$
      \item $\ehaststar \vdash \st(x^{\sigma^*}) \land \st(y^{\sigma^*}) \to \st (x*_{\sigma}y).$
%       \item $\ehast \vdash \st(x^{\sigma\to\tau}) \land \st(y^{\sigma}) \to \st (x[y])$
\item $\ehaststar \vdash \st(F^{0 \to \sigma^*}) \land \st(n^0) \to \st(F(0) * \ldots * F(n-1))$.
    \end{enumerate}
\end{lemma}
\begin{proof} Follows from the $\Tst$-axioms together  with the fact that the list recursor $L$ belongs to $\Tstar$.
\end{proof}

\begin{notation}
\begin{enumerate}
\item
Given $\underline{x}=x^{\sigma_0}_0,\ldots,x^{\sigma_{m-1}}_{m-1}$ and $\underline{i}=i^{0}_0,\ldots,i^{0}_{m-1}$ we will write
$|\tup{x}|$ for $|x^{\sigma_0}_0|,\ldots,|x^{\sigma_{m-1}}_{m-1}|$ and $(\tup{x})_{\tup{i}}$ for $(x^{\sigma_0}_0)_{i_0},\ldots,(x^{\sigma_{m-1}}_{m-1})_{i_{m-1}}$.
\item Given $\underline{x}=x^{\sigma_0}_0,\ldots,x^{\sigma_{m-1}}_{m-1}$ and $\underline{y}=y^{\sigma_0}_0,\ldots,y^{\sigma_{m-1}}_{m-1}$, we will write  $\tup{\langle \tup{x} \rangle}$
for $\langle x^{\sigma_0}_0 \rangle ,\ldots, \langle x^{\sigma_{m-1}}_{m-1} \rangle$, and $\tup{\langle \tup{x},\tup{y}\rangle}$ for $\langle x_0 , y_0\rangle ,\ldots, \langle x_{m-1},y_{m-1} \rangle$.
\item Given $\underline{x}=x^{\sigma^*_0}_0,\ldots,x^{\sigma^*_{m-1}}_{m-1}$ and $\underline{y}=y^{\sigma^*_0}_0,\ldots,y^{\sigma^*_{m-1}}_{m-1}$, we will write $\tup{x} *_{\tup{\sigma}^*} \tup{y}$ for
\[
            x_0 *_{\sigma^*_0} y_0 ,\ldots, x_{m-1} *_{\sigma^*_{m-1}} y_{m-1}.
\]
\end{enumerate}
\end{notation}

\subsection{Finite sets}

Most of the time, we will regard finite sequences as stand-ins for finite sets. In fact, we will need the notion of an element and that of one sequence being contained in another, as given in the definitions below.

\begin{dfn}\label{def:element}
For $s^{\sigma},t^{\sigma^*}$ we write $s \in_{\sigma} t$ and say that $s$ \emph{is an element of} $t$ if
\[
         \exists i < |t| (\, s =_{\sigma} (t)_i \, ).
\]
%(We will mostly write simply $\preceq$.)
For $\tup{s}^{\tup{\sigma}}=s_0^{\sigma_0},\ldots,s_{n-1}^{\sigma_{n-1}}$ and $\tup{t}^{\tup{\sigma}^*}=t_0^{\sigma^*_0},\ldots,t_{n-1}^{\sigma^*_{n-1}}$ we write $\tup{s} \in_{\tup{\sigma}} \tup{t}$ and say that $\tup{s}$ \emph{is an element of} $\tup{t}$ if
\[
        \bigwedge_{k=0}^{n-1} \, s_k \in_{\sigma_k} t_k.
\]
In case no confusion can arise, we will drop the subscript and write simply $\in$ instead of $\in_{\sigma}$ or $\in_{\sigma^*}$.
\end{dfn}

\begin{lemma} \label{elemstsetset}
$\ehaststar \vdash \st(x^{\sigma^*}) \land y \in_\sigma x \to \st(y^\sigma)$.
\end{lemma}
\begin{proof}
Follows from Lemma \ref{presstandardness}.2 and the extensionality of the $\st$-predicate (first part of the $\Tst$-axiom).
\end{proof}

\begin{dfn}\label{def:preorder}
For $s^{\sigma^*},t^{\sigma^*}$ we write $s \preceq_{\sigma} t$ and say that $s$ \emph{is contained in} $t$ if
\[
         \forall x^\sigma \, ( \, x \in s \to x \in t \, ),
\]
or, equivalently,
\[ \forall i < |s| \, \exists j < |t| \, (s)_i =_{\sigma} (t)_j. \]
%(We will mostly write simply $\preceq$.)
For $\tup{s}^{\tup{\sigma}^*}=s_0^{\sigma^*_0},\ldots,s_{n-1}^{\sigma^*_{n-1}}$ and $\tup{t}^{\tup{\sigma}^*}=t_0^{\sigma^*_0},\ldots,t_{n-1}^{\sigma^*_{n-1}}$ we write $\tup{s} \preceq_{\tup{\sigma}} \tup{t}$ and say that $\tup{s}$ \emph{is contained in} $\tup{t}$ if
\[
        \bigwedge_{k=0}^{n-1} \, s_k \preceq_{\sigma_k} t_k.
\]
%\[
%        \forall k \leq n\forall i\leq |s_k| \exists j \leq |t_k| ((s_k)_i=_{\sigma_k} (t_k)_j)
%\]
\end{dfn}

\begin{lemma} $\ehastar$ proves that $\preceq_{\sigma}$ determines a preorder on the set of objects of type $\sigma^*$. More precisely, for all $x^{\sigma^*}$ we have $x \preceq_{\sigma} x$, and for all $x^{\sigma^*},y^{\sigma^*},z^{\sigma^*}$ with $x \preceq_{\sigma} y$ and $y \preceq_{\sigma} z$, we have $x \preceq_{\sigma} z$.
\end{lemma}
\begin{proof}
Obvious.
\end{proof}

In relation to this ordering the notion of a property being \emph{upwards closed} in a variable $x$ will be of importance.
\begin{dfn}
A property $\Phi(\tup{x}^{\tup{\sigma}^*})$ is called \emph{upwards closed in $\tup{x}$} if
$\Phi(\tup{x})\land \tup{x} \preceq \tup{y} \to \Phi(\tup{y})$
and \emph{downwards closed in $\tup{x}$} if
$\Phi(\tup{x})\land \tup{y} \preceq \tup{x} \to \Phi(\tup{y})$.
\end{dfn}

\subsection{Induction and extensionality for sequences}

The aim of this subsection is to prove induction and extensionality principles for sequences. It all relies on the following lemma:

\begin{lemma}
\begin{enumerate}
\item $\ehastar \vdash \forall x^{\sigma^*} ( \, |x| = 0 \leftrightarrow x = \et_\sigma \, )$.
\item $\ehastar \vdash \forall n^0, x^{\sigma^*} ( \, |x| = Sn \leftrightarrow \exists a, y^{\sigma^*} \, ( \,  x = c(a, y) \land |y| = n \, ) \, )$.
\end{enumerate}
\end{lemma}
\begin{proof}
The right-to-left directions hold by definition of the length function $| \cdot|$. So suppose we have an element $x$ of type $\sigma^*$. Then, by the sequence axiom $\SA$, either $x = \et_\sigma$ or there are $a^\sigma, y^{\sigma^*}$ such that $x = c(a, y)$. In the latter case, $|x| = S|y| > 0$, so if $|x| = 0$, then $x = \et_\sigma$. But if $|x| =Sn$, then $x \not= \et_\sigma$ and there are $a, y$ with $x = c(a, y)$ and $|y| = n$.
\end{proof}

\begin{prop}
$\ehastar$ proves the induction schema for sequences:
\[ \varphi(\et_\sigma) \land \forall a^\sigma, y^{\sigma^*} \, ( \, \varphi(y) \to \varphi(c(a, y) \, ) \to \forall x^{\sigma^*} \, \varphi(x). \]
\end{prop}
\begin{proof}
Suppose $\varphi(\et_\sigma)$ and $\forall a^\sigma, y^{\sigma^*} \, ( \, \varphi(y) \to \varphi(c(a, y) \, )$. It now follows from the previous lemma that we can prove
\[ \forall n^0 \, \forall x^{\sigma^*} \, ( \, |x| = n \to \varphi(x) \, ) \]
by ordinary induction.
\end{proof}

One consequence is the following useful fact:

\begin{lemma}
$\ehastar$ proves that for any two elements $x^{\sigma^*},y^{\sigma^*}$ we have $|x * y| = |x| + |y|$ and
\begin{displaymath}
\begin{array}{lcll}
(x * y)_i & = & (x)_i & \mbox{if } i < |x|, \\
(x * y )_i & = & (y)_{i-|x|} & \mbox{otherwise}.
\end{array}
\end{displaymath}
Therefore it also proves that $x \preceq_{\sigma} x * y$ and $y \preceq_{\sigma} x * y$.
\end{lemma}
\begin{proof}
Easy argument using the recursive definitions of $| \cdot |$ and $*$ and the induction schema for sequences.
\end{proof}

Another consequence is the principle of extensionality for sequences. We will call two elements $x^{\sigma^*},y^{\sigma^*}$ extensionally equal, and write $x =_{e, \sigma^*} y$, if
\[ |x| =_0 |y| \land \forall i < |x| \, ( \, (x)_i =_\sigma (y)_i \, ). \]
\begin{prop} \label{extprincforseq}
$\ehastar$ proves 
\[ \forall x^{\sigma^*},y^{\sigma^*} \, ( \, x=_{e, \sigma^*} y \to x =_{\sigma^*} y \, ). \]
\end{prop}
\begin{proof}
Proof by sequence induction on $x$.

If $x =_e y$ and $x = \et_\sigma$, then $|y| = |x| = 0$. So $y = \et_\sigma$.

If $x =_e y$ and $x = c(a, x')$, then $|x| = Sn$ where $n = |x'|$. So also $|y| = Sn$ and hence $y = c(b, y')$ for some $b^\sigma, y'^{\sigma^*}$ with $|y'| = n$. Since $x =_e y$, we have $a = b$ and $x' =_e y'$. From the latter we get $x' = y'$ by induction hypothesis, so $x = c(a, x') = c(b, y') = y$.
\end{proof}

\begin{cor}
$\ehaststar$ proves
\[ \forall x^{\sigma^*} \, \st(|x|) \land \forall i < |x| \, \st((x)_i) \to \st(x). \]
\end{cor}
\begin{proof}
Suppose $x^{\sigma^*}$ is a sequence of standard length and all components $(x)_i$ are standard. Then $x' := \langle x_0, \ldots, x_{|x|-1} \rangle$ is also standard (by Lemma \ref{presstandardness}). But $x =_{e, \sigma^*} x'$, so $x = x'$ by extensionality for sequences and $\st(x)$ by extensionality of the standardness predicate.
\end{proof}

\begin{cor}
$\ehaststar$ proves the external induction axiom for sequences:
\[  \Phi(\et_\sigma) \land \forallst a^\sigma, y^{\sigma^*} \, ( \, \Phi(y) \to \Phi(c(a, y) \, ) \to \forallst x^{\sigma^*} \, \Phi(x). \]
\end{cor}
\begin{proof}
Suppose $\Phi(\et_\sigma)$ and $\forallst a^\sigma, y^{\sigma^*} \, ( \, \Phi(y) \to \Phi(c(a, y) \, )$. The idea now is to prove
\[ \forallst n^0 \, \forallst x^{\sigma^*} \, ( \, |x| = n \to \Phi(x) \, ) \]
by external induction $\IA^{\st}$, using the previous corollary to argue that if  $x = c(a, y)$ and $x^{\sigma^*}$ is standard, then both $a$ and $y$ are standard as well.
\end{proof}

\subsection{Finite sequence application}

The following operations will be crucial for what follows.

\begin{dfn}[Finite sequence application and abstraction]
If $s$ is of type $(\sigma \to \tau^*)^*$ and $t$ is of type $\sigma$, then
\[
         s[t] := (s)_0 (t) *\ldots * (s)_{|s|-1}(t):\tau^*.
\]
For every term $s$ of type $\sigma \to \tau^*$ we set 
\[ \Lambda x^\sigma.s(x):=\langle\lambda x^\sigma . s(x)\rangle: (\sigma \to \tau^*)^*. \]
\end{dfn}

The point is that these two operations act like application and abstraction, for we have
\[ (\Lambda x.s(x))[t] =_{\tau^*} (\lambda x.s(x))(t)=_{\tau^*} s(t). \]
We will often write $\tup s[\tup t]$ and $\Lambda \tup x. \tup t$; in handling these expressions, the same conventions as for ordinary application and abstraction apply (see Section 2.1).

Note that the defining equations for the sequence application and $\Lambda$-abstraction imply that finite sequence application and ordinary application are (provably) interdefinable, in the following sense: $\ehastar$ proves that for every $s: (\sigma \to \tau^*)^*$ there is a $t: \sigma \to \tau^*$ (viz., $t = \lambda x. s[x]$) such that $s[x] = t(x)$ for all $x$, as well as that for every $t: \sigma \to \tau^*$ there is an $s$ of type $(\sigma \to \tau^*)^*$ (viz., $s = \Lambda x. t(x)$) such that $s[x] = t(x)$ for all $x$.

In what follows we will need that one can define recursors $\tup{\mathcal{R}}_{\tup{\rho}}$ for each tuple of types $\tup{\rho}^* = \rho^*_0, \ldots, \rho^*_k$, such that
\begin{eqnarray*}
          \tup{\mathcal{R}}_{\tup{\rho}} (0,\tup{y},\tup{z}) &=_{\tup{\rho}^*} & \tup{y}, \\
           \tup{\mathcal{R}}_{\tup{\rho}} (n+1,\tup{y},\tup{z}) &=_{\tup{\rho}^*} & \tup{z}[n,\tup{\mathcal{R}}_{\tup{\rho}} (n,\tup{y},\tup{z})],
\end{eqnarray*}
(where $y_i$ is of type $\rho^*_i$ and $z_i$ is of type $(0 \to \rho^*_0 \to \ldots \to \rho^*_k  \to \rho^*_i)^*$). Indeed, by letting
\[
      \tup{\mathcal{R}}_{\tup{\rho}} :=\lambda n^0,\tup{y},\tup{z}. \tup{R}_{\tup{\rho}^*}(n,\tup{y},(\lambda \tup{s}^{\tup{\rho}^*}, t^0 . \tup{z} [t,\tup{s}])),
\]
where $\tup{R}_{\tup{\rho}}$ are constants for simultaneous primitive recursion as in~\cite{Kohlenbach08},
we get
\[
          \tup{\mathcal{R}}_{\tup{\rho}} (0,\tup{y},\tup{z}) =_{\tup{\rho}^*} \tup{R}_{\tup{\rho}^*}(0,\tup{y},(\lambda \tup{s}^{\tup{\rho}^*}, t^0 . \tup{z} [t,\tup{s}]))  =_{\tup{\rho}^*} \tup{y}
\]
and
\begin{eqnarray*}
           \tup{\mathcal{R}}_{\tup{\rho}} (n+1,\tup{y},\tup{z}) &=_{\tup{\rho}^*} & \tup{R}_{\tup{\rho}^*}(n+1,\tup{y},(\lambda \tup{s}^{\tup{\rho}^*}, t^0 . \tup{z} [t,\tup{s}])) \\
           &=_{\tup{\rho}^*} & (\lambda \tup{s}^{\tup{\rho}^*}, t^0 . \tup{z} [t,\tup{s}]) (\tup{R}_{\tup{\rho}^*} (n,\tup{y},  (\lambda \tup{s}^{\tup{\rho}^*}, t^0 . \tup{z} [t,\tup{s}])  ),n) \\
           &=_{\tup{\rho}^*} &  \tup{z}[n,\tup{R}_{\tup{\rho}^*} (n,\tup{y},  (\lambda \tup{s}^{\tup{\rho}^*}, t^0 . \tup{z} [t,\tup{s}])  )] \\
           &=_{\tup{\rho}^*} &  \tup{z}[n,\tup{\mathcal{R}}_{\tup{\rho}} (n,\tup{y},\tup{z})].
\end{eqnarray*}
Notice that when compared to the case of the ordinary primitive recursors $\tup{R}_{\tup{\rho}}$ we have switched the order of the arguments of $z$. This is simply to make the realizer for the interpretation of the induction schema nicer.% in the case of Herbrand realizability.

With respect to the preorder
$\preceq$ from Definition~\ref{def:preorder} the new application is monotone in the first component, in the following sense:
\begin{lemma}\label{le:herbrand:new_application} $\ehastar$ proves
      \begin{enumerate}
         \item If $s^{(\sigma\to\tau^*)^*} \preceq \tilde{s}^{(\sigma\to\tau^*)^*}$, then
                   $s[t] \preceq \tilde{s}[t]$, for all $t^{\sigma}$.
         \item If $s \preceq \tilde{s}$, then $s[\tup{t}] \preceq \tilde{s}[\tup{t}]$ for all $\tup{t}$ of suitable types.
         \item If $\tup{s} \preceq \tup{\tilde{s}}$, then $\tup{s}[\tup{t}] \preceq \tup{\tilde{s}}[\tup{t}]$ for all $\tup{t}$ of suitable types.
    \end{enumerate}
\end{lemma}
\begin{proof} We will only prove the first point, as the other two are similar. Let $i < |s[t]|$, and consider $(s[t])_i$. Since
$s[t] =_{\tau^*} (s)_0(t) * \ldots * (s)_{|s|-1}(t)$ there is $k < |s|$ and $m < |(s)_k (t)|$ such that
\[
(s[t])_i =_{\tau^*} ((s)_k (t))_m.
\]
Since $s\preceq \tilde{s}$ there is $j < |\tilde{s}|$ such that $(s)_k =_{\sigma\to\tau^*} (\tilde{s})_j$. Thus $m < |(\tilde{s})_j(t)|$ and
\[
      (s[t])_i =_{\tau^*} ((\tilde{s})_j (t))_m,
\]
and so since $\tilde{s}[t] =_{\tau^*} (\tilde{s})_0(t) * \ldots * (\tilde{s})_{|\tilde{s}|-1}(t)$ there is some
$n < |\tilde{s}[t]|$ such that
\[
       (\tilde{s}[t])_n =_{\tau^*} ((\tilde{s})_j (t))_m =_{\tau^*} (s[t])_i.
\]
\end{proof}

\begin{lemma} $\ehaststar$ proves \[         \st^{(\sigma\to\tau^*)^*}(x) \land \st^{\sigma}(y) \to \st^{\tau^*} (x[y]) \] and \[ \st^{\sigma \to \tau^*}(s) \to \st^{(\sigma \to \tau^*)^*}(\Lambda x^\sigma. s(x)). \]
\end{lemma}
\begin{proof} Follows from Lemma \ref{presstandardness}.
\end{proof}

\section{Nonstandard principles}

Semantic approaches to nonstandard analysis exploit the existence of nonstandard models of the first-order theory of the natural numbers or the reals. In fact, one may use the compactness theorem for first-order logic or the existence of suitable nonprincipal ultrafilters to show that there are extensions of the natural numbers, the reals or any other first-order structure one might be interested in, that are \emph{elementary}: that is, satisfy the same first-order sentences, even when allowing for parameters from the original structure. For the natural numbers, for instance, this means that there are structures ${}^*\NN$ and embeddings $i: \NN \to {}^*\NN$ that satisfy
\[ {}^*\NN \models \varphi(i(n_0), \ldots, i(n_k)) \Longleftrightarrow \NN \models \varphi(n_0, \ldots, n_k) \]
for all first-order formulas $\varphi(x_0, \ldots, x_k)$ and natural numbers $n_0, \ldots, n_k$. Usually, one identifies the elements in the image of $i$ with the natural numbers and calls these the \emph{standard} natural numbers, while those that do not lie in the image of $i$ are the \emph{nonstandard} natural numbers. Sometimes, one adds a new predicate $\st$ to the structure ${}^*\NN$, which is true only of the standard natural numbers. One can then use the elementarity of the embedding to show that ${}^*\NN$ is still a linear order in which the nonstandard natural numbers must be infinite (i.e., bigger than any standard natural number). The charm and power of nonstandard proofs is that one can use these infinite natural numbers to prove theorems in the nonstandard structure ${}^*\NN$, which must then be true in $\NN$ as well, as the embedding $i$ is elementary. The same applies to nonstandard extensions ${}^*\RR$ of the reals, in which there are besides infinite reals, also infinitesimals (nonstandard reals having an absolute value smaller than any positive standard real): these infinitesimals can then be used to prove theorems in analysis in ${}^*\RR$; again, one can then go on to use the elementarity of the embedding to show that they must hold in $\RR$ as well. The catch is that only \emph{first-order, internal} statements can be lifted in this way: so using nonstandard models requires some understanding about what can and what can not be expressed in first-order logic as well as some careful verifications as to whether formulas are internal.

Besides creating an interesting world in which there are infinite natural numbers and infinitesimals, nonstandard analysis also comes with some new proof principles, among which the following are the most important:
\begin{enumerate}
\item Overspill: if $\varphi(x)$ is internal and holds for all standard $x$, then $\varphi(x)$ also holds for some nonstandard $x$.
\item Underspill: if $\varphi(x)$ is internal and holds for all nonstandard $x$, then $\varphi(x)$ also holds for some standard $x$.
\item Transfer: an internal formula $\varphi$ (possibly with standard parameters) holds in ${}^*\NN$ iff it holds in $\NN$.
\end{enumerate}
Of course, transfer expresses the elementarity of the embedding. The other two principles are consequences of the fact that it is impossible to define standardness in ${}^*\NN$ using an internal formula: for if one could, then $\varphi(n)$ would hold in ${}^*\NN$, and hence in $\NN$, for every natural number $n$. This would imply that $\forall x \, \varphi(x)$ holds in $\NN$ and hence in ${}^*\NN$ as well; but that contradicts the existence of nonstandard elements in ${}^*\NN$.

In the remainder of this section, we will discuss these principles in more detail, for two related reasons. First of all, they will provide us with three benchmarks with which we will be able to measure the success of the different interpretations. Also, because they have some nontrivial consequences (especially in the intuitionistic context), discussing these will give us some important clues as to how any interpretation of nonstandard analysis will have to look like.

\begin{remark}
Unless we state otherwise, the principles we will subsequently introduce in this paper may have additional parameters besides those explicitly shown. Also recall that we follow Nelson's convention in using small Greek letters to denote internal formulas and capital Greek letters to denote formulas which can be external.
\end{remark}

\subsection{Overspill}

When formalised in $\ehaststar$, overspill (in type 0) is the following statement:
\[ \OS_0: \forallst x^0 \, \varphi(x) \to \exists x^0 \, ( \, \lnot \st(x) \land \varphi(x) \, ). \]

\begin{prop} {\rm \cite{palmgren98}}
In $\ehaststar$, the principle $\OS_0$ implies the existence of nonstandard natural numbers,
\[ \ENS_0: \exists x^0 \, \lnot \st(x), \]
as well as:
\[ \LLPO_0: \forallst x^0, y^0 \, ( \, \varphi(x) \lor \psi(y) \, ) \to \forallst x^0 \, \varphi(x) \lor \forallst y^0 \, \psi(y) . \]
\end{prop}
\begin{proof}
$\ENS_0$ follows trivially from $\OS_0$ by taking for $\varphi(x)$ some trivially true formula (for instance, $x = x$).

If $\forallst x^0, y^0 \, ( \, \varphi(x) \lor \psi(y) \, )$, then one can use external induction to prove:
\[ \forallst n^0 \, ( \, \forall m \leq n \, \varphi(m) \lor \forall m \leq n \, \psi(m) \, ). \]
By applying overspill to this statement, we see that it holds for some nonstandard $n$. Using external induction again, we can prove that a nonstandard natural number must be bigger than any standard natural number. Hence $\forallst x^0 \, \varphi(x) \lor \forallst y^0 \, \psi(y)$.
\end{proof}

Of course, overspill can be formulated for all types:
\[ \OS: \forallst x^\sigma \, \varphi(x) \to \exists x^\sigma \, ( \, \lnot \st(x) \land \varphi(x) \, ). \]
But, actually, the interpretations that we will discuss will not only verify this principle, but also a far-reaching generalization of it, \emph{viz.} a higher-type version of Nelson's idealization principle \cite{nelson77}:
\[ \I: \forallst x^{\sigma^*} \exists y^\tau \forall x' \in_{\sigma} x \, \varphi(x', y) \to \exists y^\tau \forallst x^\sigma \, \varphi(x, y). \]

\begin{prop} {\rm \cite{palmgren98}} In $\ehaststar$, the idealization principle $\I$ implies overspill, as well as the statement that for every type $\sigma$ there is a nonstandard sequence containing all the standard elements of that type:
\[ \USEQ: \exists y^{\sigma^*} \, \forallst x^\sigma \, x \in_{\sigma} y. \]
\end{prop}
\begin{proof}
$\OS$ for an internal formula $\psi(y)$ follows from $\I$ by taking $\varphi(x, y) \, :\equiv \, y \not= x \land \psi(y),$ while $\USEQ$ follows by taking
$\varphi(x, y) \, :\equiv \, x \in_{\sigma} y.$
\end{proof}

\begin{prop}\label{prop:LLPO} In $\ehaststar$, the idealization principle $\I$ implies the existence of nonstandard elements of any type,
\[ \ENS: \exists x^\sigma \, \lnot \st(x), \]
as well as $\LLPO$ for any type:
\[ \LLPO: \forallst x^\sigma, y^\sigma \, ( \, \varphi(x) \lor \psi(y) \, ) \to \forallst x^\sigma \, \varphi(x) \lor \forallst y^\sigma \, \psi(y) . \]
\end{prop}
\begin{proof}
The first statement is obvious, so we concentrate on the second.

Suppose $\forallst x^\sigma, y^\sigma \, ( \, \varphi(x) \lor \psi(y) \, )$. Then one easily proves by external sequence  induction and by taking $v := u$ that
\[ \forallst u^{\sigma^*} \, \exists v^{\sigma^*} \, \forall u' \in u \, \big( \, u' \in v \land \, ( \, \forall v' \in v \, \varphi(v') \lor \forall v' \in v \, \psi(v') \, ) \, \big). \]
By applying idealization to this statement we obtain
\[ \exists v^{\sigma^*} \, \forallst u^\sigma \, \big( \, u \in v \land \, ( \, \forall v' \in v \, \varphi(v') \lor \forall v' \in v \, \psi(v') \, ) \, \big), \]
from which $\LLPO$ follows.
\end{proof}

Classically, idealization is equivalent to its dual, which we have dubbed the \emph{realization principle} (intuitionistically, things are not so clear):
\[ \R: \forall y^\tau \existsst x^\sigma \, \varphi(x, y) \to \existsst x^{\sigma^*} \forall y^\tau \exists x' \in x \, \varphi(x', y). \]
As it turns out, both our interpretations will eliminate this principle as well. Actually, both interpretations for constructive nonstandard analysis eliminate the stronger \emph{nonclassical realization principle}:
\[ \NCR: \forall y^\tau \existsst x^\sigma \, \Phi(x, y) \to \existsst x^{\sigma^*} \forall y^\tau \exists x' \in x \, \Phi(x', y), \]
where $\Phi(x, y)$ can be any formula. This is quite remarkable, as $\NCR$ is incompatible with classical logic (hence the name) in that one can prove:

\begin{prop}
In $\ehaststar$, the nonclassical realization principle $\NCR$ implies the undecidability of the standardness predicate:
\[ \lnot \forall x^\sigma \, ( \, \st(x) \lor \lnot \st(x) \, ). \]
\end{prop}
\begin{proof}
Assume that standardness would be decidable. Then we would have
\[ \forall y^\sigma \existsst x^\sigma \, ( \, \st(y) \to x = y \, ). \]
Applying $\NCR$ to this statement yields:
\[ \existsst x^{\sigma^*} \forall y^\sigma \exists x' \in x \, ( \, \st(y) \to x' = y \, ), \]
which is the statement that there are only finitely many standard elements of type $\sigma$. This is clearly absurd.
\end{proof}

\subsection{Underspill}

Underspill (in type 0) is the following statement:
\[ \US_0: \forall x^0 \, ( \, \lnot \st(x) \to \varphi(x) \, ) \to \existsst x^0 \, \varphi(x). \]
In a constructive context it has the following nontrivial consequence (compare \cite{avigadhelzner02}):

\begin{prop} \label{US_0impliesMP_0}
In $\ehaststar$, the underspill principle $\US_0$ implies
\[ \MP_0: \big( \, \forallst x^0 \, ( \, \varphi(x) \lor \lnot \varphi(x) \, ) \land \lnot \lnot \existsst x^0 \varphi(x) \, \big) \to \existsst x^0 \varphi(x). \]
In particular, $\ehaststar + \US_0 \vdash \lnot \lnot \st^0 (x) \to \st^0(x)$.
\end{prop}
\begin{proof}
We reason in $\ehaststar + \US_0$. Suppose $\forallst x^0 \, ( \, \varphi(x) \lor \lnot \varphi(x) \, )$ and $\lnot \lnot \existsst x^0 \varphi(x)$. \

Since the latter is intuitionistically equivalent to $\lnot \forallst x^0 \lnot \varphi(x)$, this means that for every infinite natural number $\omega$, we have $\lnot \forall x \leq \omega \lnot \varphi(x)$. In other words, we have
\[ \forall \omega^0 \, ( \, \lnot \st(\omega) \to \lnot \forall x \leq \omega \lnot \varphi(x) \, ). \]
So by $\US_0$ we have
\[ \existsst y \, \lnot \forall x \leq y \lnot \varphi(x), \]
which implies $\existsst x \, \varphi(x)$ by decidability of $\varphi(x)$ for standard values of $x$.

As a special case we have $\lnot \lnot \st^0 (x) \to \st^0(x)$, because $\st^0(x)$ is equivalent to $\existsst y^0 \, ( \, x =_0 y \,)$ and equality of objects of type 0 is decidable.
\end{proof}

Also underspill has a direct generalization to higher types:
\[ \US: \forall x^\sigma \, ( \, \lnot \st(x) \to \varphi(x) \, ) \to \existsst x^\sigma \, \varphi(x). \]
A natural question is whether this implies a version of Markov's Principle for all types. Our suspicion is that this is not the case, but we were unable to prove this.

\subsection{Transfer}

Following Nelson \cite{nelson77}, the transfer principle is usually formulated as follows:
\[ \TPA: \forallst \tup t \, ( \, \forallst x \,  \varphi(x, \tup t) \to \forall x \, \varphi(x, \tup t) \, ). \]
(Here, for once, we do not allow parameters: so it is important that $x$ and $\tup t$ include all free variables of the formula $\varphi$.) This is classically, but not intuitionistically, equivalent to the following:
\[ \TPE: \forallst \tup t \, ( \, \exists x \,  \varphi(x, \tup t) \to \existsst x \, \varphi(x, \tup t) \, ), \]
where, once again, we do not allow parameters.

It turns out that interpreting transfer is very difficult, especially in a constructive context (in fact, Avigad and Helzner have devoted an entire paper \cite{avigadhelzner02} to this issue). There are, at least, the following three problems:
\begin{enumerate}
\item Transfer principles together with overspill imply instances of the law of excluded middle, as was first shown by Moerdijk and Palmgren in \cite{moerdijkpalmgren97}. In our setting we have:
\begin{prop}
\begin{enumerate}
\item In $\ehaststar$, the combination of $\ENS_0$ and $\TPA$ implies the law of excluded middle for all internal arithmetical formulas.
\item In $\ehaststar$, the combination of $\USEQ$ and $\TPA$ implies the law of excluded middle for all internal formulas.
\end{enumerate}
\end{prop}
\begin{proof}
Ad (a): assume $\ENS_0$ and $\TPA$. We show that every internal arithmetic formula $\varphi$ is decidable by induction on the number of internal quantifiers in $\varphi$. Atomic formulas are decidable anyway, so the base case is easy.

If $\varphi(x^0, \tup t)$ is an internal formula which is decidable and arithmetical (and with all free variables shown), then one can use internal induction to show
\[ \forall y^0 \, ( \, \exists x \leq y \, \lnot \varphi(x, \tup t) \lor \forall x \leq y \, \varphi(x, \tup t) \, ). \]
Let $\tup t$ be some arbitrary standard value and let $y$ be some infinite natural number $\omega$, using $\ENS_0$. Then we either have $\exists x \leq \omega \, \lnot \varphi(x, \tup t)$ and in particular $\exists x \, \lnot \varphi(x, \tup t)$ or we have $\forall x \leq \omega \, \varphi(x, \tup t)$ and in particular $\forallst x \, \varphi(x, \tup t)$. So:
\[ \forallst \tup t \, ( \, \exists x \, \lnot \varphi(x, \tup t) \lor \forallst x \, \varphi(x, \tup t) \, ). \]
Applying $\TPA$ once we get
\[ \forallst \tup t \, ( \, \exists x \, \lnot \varphi(x, \tup t) \lor \forall x \, \varphi(x, \tup t) \, ) \]
and applying it another time we get
\[ \forall \tup t \, ( \, \exists x \, \lnot \varphi(x, \tup t) \lor \forall x \, \varphi(x, \tup t) \, ). \]
This completes the induction step.

In (b) we argue similarly. First, we use the extensionality principles for functions and sequences to eliminate all equality predicates at higher types in favour of equalities at type 0. This makes all atomic formulas decidable. Since decidable formulas are closed under all propositional connectives, this leaves the case of the quantifiers. So suppose $\varphi(x^\sigma, \tup t)$ is a internal formula which is decidable and let $u$ be a sequence containing all standard elements of type $\sigma$ (using $\USEQ$). Then we have:
\[ \exists u' \in u \, \lnot \varphi(u', \tup t) \lor \forall u' \in u \, \varphi(u', \tup t). \]
In the former case it holds that $\exists x^\sigma \, \lnot \varphi(x, \tup t)$ and in the latter that $\forallst x^\sigma \varphi(x, \tup t)$. And from here the argument proceeds as before.
\end{proof}
\item As Avigad and Helzner observe in \cite{avigadhelzner02}, also the combination of transfer principles with underspill results in a system which is no longer conservative over Heyting arithmetic. More precisely, adding $\US_0$ and $\TPA$, or $\US_0$ and $\TPE$, to $\ehaststar$ results in a system which is no longer conservative over Heyting arithmetic $\HA$. The reason is that there are quantifier-free formulas $A(x)$ such that
\[ \HA \not\vdash \lnot \lnot \exists x \, A(x) \to \exists x \, A(x). \]
Since one can prove a version of Markov's Principle in $\ehaststar + \US_0$, adding either $\TPA$ or $\TPE$ to it would result in a nonconservative extension of $\HA$ (and hence of $\ehastar$). We refer to \cite{avigadhelzner02} for more details.
\item The last point applies to functional interpretations only. As is well-known, in the context of functional interpretations the axiom of extensionality always presents a serious problem and when developing a functional interpretation of nonstandard arithmetic, the situation is no different. Now, $\ehaststar$ includes an internal axiom of extensionality (as it is part of $\ehastar$), but for the functional interpretation that we will introduce in Section 5 that will be harmless. What will be very problematic for us, however, is the following version of the axiom of extensionality: if for two elements $f, g$ of type $\sigma_1 \to (\sigma_2 \to \ldots \to 0))$, we define
\[ f =^{\st} g \, :\equiv \, \forallst x_1^{\sigma_1}, x^{\sigma_2}_2, \ldots \, ( \, f\tup x =_0 g\tup x \, ), \]
then extensionality formulated as
\[ \forallst f \, \forallst x, y \, ( \, x =^{\st} y \to fx =^{\st} fy \, ) \]
will have no witness definable in $\ZFC$. But that means that also $\TPA$ can have no witness definable in $\ZFC$: for in the presence of $\TPA$ both versions of extensionality are equivalent.
\end{enumerate}

One way out of this quandary, which is strongly suggested by the last point and is the route taken in most sources (beginning with \cite{moerdijk95}), is to have transfer not as a principle, but as a \emph{rule}. As we will see, this turns out to be feasible. In fact, we will have two transfer rules (which are not equivalent, not even classically):
\[ \begin{array}{ccc}
\infer[\TRA]{\forall x \, \varphi(x)}{\forallst x \, \varphi(x)} & & \infer[\TRE]{\existsst x \, \varphi(x)}{\exists x \, \varphi(x)}
\end{array} \]
(This time round there are no special requirements on the parameters of $\varphi$.)

\begin{remark}
In this section we have also explored several connections between nonstandard principles. More principles will be introduced below and we will prove one more implication (see Proposition \ref{USfromRHGMP}). We have not tried to determine the precise relationships between these principles -- in particular, we do not know precisely which principles follow and do not follow over $\ehaststar$ from combinations of other principles. In fact, we believe that mapping these connections would be an interesting research project.
\end{remark}
\section{Herbrand realizability}

In this section we will introduce a new realizability interpretation, which will allow us to prove our first consistency and conservation results in the context of $\ehaststar$. Our treatment here will be entirely proof-theoretic; for a semantic approach towards Herbrand realizability, see \cite{berg12}.

\subsection{The interpretation}

The interpretation works by associating to every formula $\Phi(\tup x)$ a formula $\Psi(\tup t,\tup x)$, also denoted by $\tup t \hr \Phi(\tup x)$, where $\tup t$ is a tupe of new variables \emph{all of which are of sequence type}, determined solely by the logical form of $\Phi(\tup x)$. The soundness proof will then involve showing that for every formula $\Phi(\tup x)$ of $\ehaststar$ with $\ehaststar \vdash \Phi(\tup x)$ there is an appropriate tuple $\tup t$ of terms from $\Tstar$ such that $\ehaststar \vdash \tup t \hr \Phi(\tup x)$.

The idea of the interpretation is that we interpret internal quantifers uniformly and that we do not attempt to give them any computational content. In a sense, the only predicate to which we will assign any computational content is the standardness predicate $\st$: to realize $\st^\sigma(x)$, however, it suffices to provide a nonempty finite list of terms of type $\sigma$, one of which will have to be equal to $x$. Therefore to realize a statement of the form $\existsst x^\sigma \, \Phi(x)$ one only needs to provide a finite list $\langle y_0, \ldots, y_n \rangle$ and to make sure that $\Phi(y_i)$ is realized for some $i \leq n$ (which is like giving a Herbrand disjunction; hence the name ``Herbrand realizability'').

The precise definition is as follows:
\begin{dfn}[Herbrand realizability for $\ehaststar$]\label{d:herRel}
\begin{align*}
[]&\hr \phi  & &:\equiv& &\phi\quad \text { for an internal atomic formula $\phi$}, \\
%t&\hr \st(s) & &:\equiv& &t=\langle t_0, \ldots, t_n\rangle\text{ and $s=t_i$ for some $i\leq n$},\\
s&\hr \st(x) & &:\equiv& & x \in s, \\
\tup s, \tup t&\hr (\Phi\vee\Psi) & &:\equiv& &\tup s\hr\Phi\vee \tup t\hr\Psi,\\
 \tup s, \tup t & \hr (\Phi \wedge \Psi) & &:\equiv& &\tup {s}\hr\Phi \, \land \, \tup {t}\hr\Psi,\\
\tup s&\hr (\Phi \rightarrow \Psi) & &:\equiv& & \forallst \tup t\ (\ \tup t\hr\Phi
 \rightarrow\tup s \, [\tup t \, ] \hr\Psi),\\
\tup s&\hr \exists x \, \Phi(x) & &:\equiv& & \exists x\ (\tup s\hr\Phi(x)),\\
\tup s&\hr \forall x \, \Phi(x) & &:\equiv& & \forall x\ (\tup s\hr\Phi(x)),\\
s, \tup t &\hr \existsst x \, \Phi(x) & &:\equiv & & \exists s' \in s \, \big( \, \tup t \hr \Phi(s') \, \big),\\
\tup s &\hr \forallst x \, \Phi(x) & &:\equiv& & \forallst x\ \big(\tup s \, [x] \hr \Phi(x)\big).
\end{align*}
\end{dfn}

Before we show the soundness of the interpretation, we first prove some easy lemmas:

\begin{dfn}[The $\exists^{\st{}}$-free formulas]
We call a formula (in the language of $\ehaststar$) $\exists^{\st{}}$-free, if it is built up from atomic formulas (including $\bot$) using the connectives $\wedge$, $\lor$, $\to$ and the quantifiers $\exists x$, $\forall x$ and $\forallst x$. Alternatively, one could say that these are the formulas in which $\st$ and $\existsst$ do not occur. We denote such formulas by $\Phi_{\not\exists^{\st}}$.
\end{dfn}

\begin{lemma}[Interpretation of $\exists^{\st{}}$-free formulas]\label{l:IntOfEFFormulas} All the interpretations $\tup t \hr \Phi(\tup x)$ of formulas $\Phi(\tup x)$ of $\ehaststar$ are $\exists^{\st{}}$-free. In addition, every $\exists^{\st{}}$-free formula is interpreted by itself. Hence the interpretation is idempotent.
\end{lemma}

The following lemma will be crucial for what follows:

\begin{lemma}[Realizers are provably upwards closed]\label{l:hrealizersext} The formula $\tup t \hr \Phi(\tup x)$ is provably upwards closed in $\tup t$, that is: 
\[ \ehaststar \vdash \tup s\hr\Phi \land \tup s \preceq \tup t \to \tup t\hr\Phi. \]
\end{lemma}
\begin{proof}
By induction on the structure of $\Phi$, using the monotonicity of the new application in the first component in the clauses for $\to$ and $\forallst$.
\end{proof}

\begin{thm}[Soundness of Herbrand realizability]
Let $\Phi$ be an arbitrary formula of $\ehaststar$ and $\Delta_{\not\exists^{\st}}$ be an arbitrary set of $\existsst$-free sentences. Whenever
\[\ehaststar + \Delta_{\not\exists^{\st}} \quad \vdash\quad\Phi(\tup x),\]
then one can extract from the formal proof closed terms $\tup t$ in $\Tstar$, such that
\[\ehaststar + \Delta_{\not\exists^{\st}} \quad\vdash\quad \tup t \hr \Phi(\tup x).\]
\end{thm}
\begin{proof}
As for the logical axioms and rules, the differences with the usual soundness proof of modified realizability for $\eha$ (as in \cite{Troelstra73} or Theorem 5.8 in~\cite{Kohlenbach08}) are
\begin{enumerate}
\item[(a)] that we require the realizing terms to be closed,
\item[(b)] that we have a nonconstructive interpretation of disjunction, and 
\item[(c)] that we interpret the quantifiers in a uniform fashion.
\end{enumerate}
Therefore one has to make the following modifications:
\begin{enumerate}
\item The contraction axiom $A \lor A \to A$ is realized by $\Lambda \tup x, \tup y. \tup x * \tup y$, using that the collection of realizers is provably upwards closed.
\item The weakening axiom $A \to A \lor B$ is realized by $\Lambda \tup x.\tup x, {\cal O}$.
\item The permutation axiom $A \lor B \to B \lor A$ is realized by $\Lambda \tup x, \tup y.\tup y, \tup x$.
\item The axioms of $\forall$-elimination $\forall x\Phi(x)\rightarrow\Phi(t)$ and $\exists$-introduction $\Phi(t)\rightarrow\exists x\Phi(x)$ are realized by the identity tuple $\Lambda \tup x\ .\ \tup x$.
\item The expansion rule $\frac{A \to B}{A \lor C \to A \lor C}$: if $\tup t \hr A \to B$, then $\Lambda \tup x, \tup y. \tup t[\tup x], \tup y$ is a Herbrand realizer of $A \lor C \to A \lor C$.
\item The $\forall$-introduction rule $\frac{\Phi\rightarrow\Psi(x)}{\Phi\rightarrow\forall x\Psi(x)}$ is interpreted, because $\tup s\hr (\Phi\rightarrow\Psi(x))$ implies $\tup s \hr  (\Phi\rightarrow\forall x\Psi(x))$: for if $\tup t \hr \Phi$ and $\st(\tup t)$, then $\tup s[\tup t]  \hr \Psi(x)$ and therefore $\tup s[ \tup t] \hr  \forall x\Psi(x)$.
\item The $\exists$-introduction rule $ \frac{\Phi(x)\rightarrow\Psi}{\exists x\Phi(x)\rightarrow\Psi}$ is interpreted, because $\tup s \hr (\Phi(x)\rightarrow\Psi)$ implies $\tup s \hr (\exists x\Phi(x)\rightarrow\Psi)$: for if $\tup t \hr \exists x\Phi(x)$ and $\st(\tup t)$, then there is an $x$ such that $ \tup t\hr \Phi(x)$, from which it follows that $\tup s[\tup t] \hr \Psi$. 
\end{enumerate}
The sentences from $\Delta_{\not\exists^{\st}}$ and the axioms of $\ehastar$, including $\SA$ and the defining axioms for equality, successor, combinators and recursion, are $\exists^{\st{}}$-free and therefore realized by themselves. Therefore it remains to show the soundness of the following rules and axioms:
\begin{enumerate}
\item The external quantifier axioms $\EQ$: both directions in $\existsst x \Phi(x) \leftrightarrow  \exists x (\st(x)\wedge\Phi(x))$ are interpreted by the the identity. In $\forallst x \Phi(x) \leftrightarrow   \forall x (\st(x)\rightarrow\Phi(x))$ the right-to-left direction is realized by $\Lambda \tup s, x. \tup s \, [\langle x \rangle]$, while the left-to-right direction is realized by $\Lambda \tup s, x. \tup s[x_0] * \ldots * \tup s[x_{|x|-1}]$.
\item The axiom schemes $\Tst$: the principle $\st(x) \land x = y \to \st(y)$ is realized by the identity, while $\st(t)$ is realized by $\langle t \rangle$. In addition, $\st(f) \land \st(x) \to \st(fx)$ is realized by $\Lambda f, x. \langle f_i(x_j) \rangle_{i < |f|, j < |x|}$. 
\item The induction schema $\IA^{\st{}}$:
suppose $\tup s \hr \Phi(0)$ and $\tup t \hr \forallst n (\Phi(n)\rightarrow\Phi(n+1))$, with $\st(\tup s)$ and $\st(\tup t)$. Then $\ehaststar$ proves by external induction that for standard natural numbers $n$, the term $\tup {\mathcal R} (n,\tup s,\atup t)$ is standard and $\tup {\mathcal R} (n,\tup s,\atup t) \hr \Phi(n)$. Therefore $\Lambda \tup x,\tup y, n\ .\ \tup {\mathcal R} (n,\tup x,\atup y) \hr \IA^{\st{}}$.
\end{enumerate}
\end{proof}

\subsection{The characteristic principles of Herbrand realizability}

In this section we will prove that $\HAC$ (the herbrandized axiom of choice), $\HIP_{\not\exists^{\st}}$ (the herbrandized independence of premise principle for $\exists^{\st}$-free formulas) and $\NCR$ axiomatize Herbrand realizability:
\begin{enumerate}
\item $\HAC$ : \[
     \forallst x \existsst y \, \Phi(x,y) \to \existsst F \forallst x \exists y \in F(x) \,  \Phi (x,y),
            \]
            where $\Phi(x,y)$ can be any formula.
           If $\Phi(x,y)$ is upwards closed in $y$, then this is equivalent to
           \[
               \forallst x \existsst y  \, \Phi(x,y) \to \existsst F \forallst x \, \Phi (x,F(x)).
            \]
\item $\HIP_{\not\exists^{\st}}$: \[
    \big( \Phi \to\existsst y \, \Psi(y)\big)\rightarrow \existsst y \, \big( \Phi \to \exists y' \in y \, \Psi(y')\big),
            \]
            where $\Phi$ has to be an $\existsst$-free formula and $\Psi(y)$ can be any formula.
           If $\Psi(y)$ is upwards closed in $y$, then this is equivalent to
           \[
        \big( \Phi \to\existsst y\Psi(y)\big)\rightarrow \existsst y \, \big(\Phi \to \Psi(y)\big).  
            \]
\item $\NCR$: \[
                \forall x \existsst y \,  \Phi (x,y) \to \existsst y \, \forall x \, \exists y' \in y \, \Phi(x,y') ,
            \]
            where $\Phi(x, y)$ can be any formula.
           If $\Phi(x,y)$ is upwards closed in $y$, then this is equivalent to
           \[
               \forall x \existsst y \, \Phi (x,y) \to \existsst y \forall x  \, \Phi(x,y) .
            \]
\end{enumerate}

\begin{thm}[Characterization theorem for Herbrand realizability]
\qquad 
\begin{enumerate}
\item For any instance $\Phi$ of $\HAC$, $\HIP_{\not\exists^{\st}}$ or $\NCR$, there are closed terms $\tup t$ in $\Tstar$ such that \[ \ehaststar \vdash \tup t \hr \Phi. \]
\item For any formula $\Phi$ of $\ehaststar$, we have
\[ \ehaststar + \HAC + \HIP_{\not\exists^{\st}} +\NCR \vdash \Phi \leftrightarrow \existsst \tup x \, ( \tup x \hr \Phi). \]
\end{enumerate}
\end{thm}
\begin{proof}
Soundness of $\HAC$: If $\tup r = s, \tup t$ and $\tup r \hr \forallst x\existsst y \Phi(x,y)$, then for every standard $x$ there is an $s' \in s[x]$ such that $\tup t[x] \hr \Phi(x, s')$. Hence $\langle \lambda x. s[x] \rangle, \tup t \hr \existsst F \forallst x \exists y \in F(x) \,  \Phi (x,y)$. So $\HAC$ is realized by $\Lambda x, \tup y.\langle \lambda z. x[z] \rangle, \tup y$.

Soundness of $\HIP_{\not\exists^{\st}}$: Suppose $\Phi$ is $\existsst$-free, $\tup r = s, \tup t$ and $\tup r \hr \Phi \to \existsst y \Phi(y)$. This means that if $\Phi$ would hold, then there would be an $s' \in s$ such that $\tup t \hr \Psi(s')$. Hence $\langle s \rangle, \tup t \hr \existsst y (\Phi \to \exists y' \in y \Psi(y'))$. So  $\HIP_{\not\exists^{\st}}$ is realized by $\Lambda x, \tup y.\langle x \rangle, \tup y$.

In a similar manner one checks that also $\NCR$ is realized by $\Lambda x, \tup y.\langle x \rangle, \tup y$. This completes the proof of item 1.

Item 2 one proves by induction on the logical structure of $\Phi$. We discuss implication as an illustrative case, as it is by far the hardest, and leave the other cases to the reader. We reason in $\ehaststar + \HAC + \HIP_{\not\exists^{\st}} +\NCR$. By induction hypothesis, we have that $\Phi \leftrightarrow \existsst \tup t (\tup t \hr \Phi)$ and $\Psi \leftrightarrow \existsst \tup s (\tup s \hr \Psi)$ and therefore
\[ \Phi \to \Psi \]
is equivalent to
\[ \existsst \tup t (\tup t \hr \Phi) \to \existsst \tup s (\tup s \hr \Psi), \]
which in turn is equivalent to:
\[ \forallst \tup t \, \big( \, \tup t \hr \Phi \to \existsst \tup s \, ( \, \tup s \hr \Psi \, ) \, \big). \]
Because $\tup t \hr \Phi$ is $\existsst$-free and $\tup s \hr \Psi$ is upwards closed in $\tup s$, we can use $\HIP_{\not\exists^{\st}}$ to rewrite this as:
\[ \forallst \tup t \, \existsst \tup s \, \big( \, \tup t \hr \Phi \to \tup s \hr \Psi \, \big). \]
As $\tup t \hr \Phi \to \tup s \hr \Psi$ is upwards closed in $\tup s$ and finite sequence application and ordinary application are interdefinable, we can use $\HAC$ to see that this is equivalent to:
\[ \existsst \tup s \, \forallst \tup t \, \big( \, \tup t \hr \Phi \to \tup s[\tup t] \hr \Psi \, \big), \]
which is precisely the meaning of $\existsst \tup s \, ( \, \tup s \hr (\Phi \to \Psi) \, )$.
\end{proof}

\begin{thm}[Main theorem on program extraction by $\hr$] Let $\forallst x \existsst y \Phi(x, y)$ be a sentence of $\ehaststar$ and $\Delta_{\not\exists^{\st}}$ be an arbitrary set $\existsst$-free sentences. Then the following rule holds
\begin{align*}
\ehaststar + \HAC + \HIP_{\not\exists^{\st}} + \NCR + \Delta_{\not\exists^{\st}} & \vdash \forallst x \, \existsst y \, \Phi(x, y) \Rightarrow \\
\ehaststar + \HAC + \HIP_{\not\exists^{\st}} + \NCR + \Delta_{\not\exists^{\st}} & \vdash \forallst x \, \exists y \in  t(x) \, \Phi(x, y),
\end{align*}
where $t$ is a closed term from $\Tstar$ which is extracted from the original proof using Herbrand realizability.

In the particular case where both $\Phi(x, y)$ and $\Delta_{\not\exists^{\st}}$ are internal, the conclusion yields
\[ \ehastar + \Delta_{\not\exists^{\st}} \vdash \forall x \,  \exists y \in t(x) \, \Phi(x, y). \]
If we assume that the sentences from $\Delta_{\not\exists^{\st}}$ are not just internal, but also true (in the set-theoretic model), the conclusion implies that $\forall x \,  \exists y \in t(x) \, \Phi(x, y)$ must be true as well.
\end{thm}
\begin{proof}
If
\begin{align*}
\ehaststar + \HAC + \HIP_{\not\exists^{\st}} + \NCR + \Delta_{\not\exists^{\st}} & \vdash \forallst x \, \existsst y \, \Phi(x, y),
\end{align*}
then the soundness proof yields terms $r, \tup s$ such that
\[ \ehaststar + \Delta_{\not\exists^{\st}} \vdash r, \tup s \hr \forallst x \, \existsst y \, \Phi(x, y). \]
Since $r, \tup s \hr \forallst x \, \existsst y \, \Phi(x, y)$ is by definition $\forallst x \, \exists y \in r[x] ( \, \tup s \hr \Phi(x, y) \,)$, the first statement follows by taking $t = \lambda x. r[x]$.

If both $\Phi(x, y)$ and $\Delta_{\not\exists^{\st}}$ are internal, then $\tup s$ is empty and we get
\[ \ehaststar + \Delta_{\not\exists^{\st}} \vdash \forallst x \,  \exists y \in t(x) \, \Phi(x, y). \]
By internalizing the statement, we obtain
\[ \ehastar + \Delta_{\not\exists^{\st}} \vdash \forall x \,  \exists y \in t(x) \, \Phi(x, y). \]
Since all the axioms of $\ehastar$ are true, this implies that $\forall x \,  \exists y \in t(x) \, \Phi(x, y)$ will be true, whenever $\Delta_{\not\exists^{\st}}$ is.
\end{proof}

\subsection{Discussion}

The main virtue of Herbrand realizability may be that it points one's attention to principles like $\HAC$ and $\NCR$ and that it gives one a simple proof of their consistency. However, as a method for eliminating nonstandard principles from proofs, Herbrand realizability has serious limitations. It does eliminate the realization principle $\R$ (it even eliminates the nonclassical principle $\NCR$), but overspill, the idealization principle $\I$ and the transfer rules are $\exists^{\st}$-free and therefore simply passed to the verifying system. Even worse, the underspill principle $\US_0$ does not have a computable Herbrand realizer:
\begin{prop}
$\MP_0$ and $\US_0$ do not have computable Herbrand realizers.
\end{prop}
\begin{proof}
It is well-known that there can be no computable function witnessing the modified realizability interpretation of Markov's principle, because its existence would imply the decidability of the halting problem. A similar argument shows that $\MP_0$ does not have a computable Herbrand realizer: Kleene's $T$-predicate $T(e, x, n)$ is primitive recursive and hence 
\[ \ehaststar \vdash \forallst e, x, n \big( \, T(e,x, n) \lor \lnot T(e, x, n) \, \big). \]
So if $\MP_0$ would have a computable realizer, then so would 
\[ \forallst e \, \big( \, \lnot\lnot \existsst n\,  T(e,e, n) \to \existsst n \, T(e, e, n) \, \big). \]
But if $t$ would be such a realizer, we could decide the halting problem by checking $T(e, e, n)$ for all $n \in t[e]$.

Since $\US_0$ implies $\MP_0$ (see Proposition \ref{US_0impliesMP_0}), it follows that $\US_0$ does not have a computable realizer either.
\end{proof}

In the next section, we will show that these problems can be overcome by moving from realizability to, more complicated, functional interpretations.
%
%
%
%%%%%%%%%%%%%%%%%%%%%%%%%%%%%%%%%%%%
\section{A functional interpretation for $\ehaststar$}

In this section we will introduce and study a functional interpretation for $\ehaststar$.

\subsection{The interpretation}

The basic idea of the $\D$-interpretation (the nonstandard Dialectica interpretation) is to associate to every formula $\Phi (\tup{a})$ a new formula $\Phi(\tup{a})^{\D}\equiv\existsst \tup{x}\forallst \tup{y} \, \phi_{\D} (\tup{x},\tup{y},\tup{a})$
%(with the same free variables $\tup{a}$)
such that 
\begin{enumerate}
\item all variables in $\tup x$ are of sequence type and
\item $\phi_{\D} (\tup{x},\tup{y},\tup{a})$ is upwards closed in $\tup{x}$.
\end{enumerate}
We will interpret the standardness predicate $\st^{\sigma}$ similarly to the case for Herbrand realizability: For a realizer for the interpretation of $\st^{\sigma}(x)$ we will require a standard finite list $\langle y_0,\ldots,y_n \rangle$ of candidates, one of which must be equal to $x$.

\begin{dfn}[The $\D$-interpretation for $\ehaststar$]
We associate to every formula $\Phi (\tup{a})$ in the language of $\ehaststar$ (with free variables among $\tup{a}$) a formula $\Phi(\tup{a})^{\D}\equiv\existsst \tup{x}\forallst \tup{y} \, \phi_{\D} (\tup{x},\tup{y},\tup{a})$ in the same language (with the same free variables) by:
%s.t. $\phi_{\D} (x,y)$ is upwards closed in $x$.
    \begin{itemize}
      \item[(i)] $\phi(\tup{a})^{\D}:\equiv \phi_{\D}(\tup{a}):\equiv \phi(\tup{a})$ for internal atomic formulas $\phi(\tup{a})$,
      \item[(ii)] $\st^{\sigma}(u^{\sigma})^{\D}:\equiv \existsst x^{\sigma^*} u \in_{\sigma} x$.
    \end{itemize}
Let $\Phi(\tup{a})^{\D}\equiv\existsst \tup{x}\forallst \tup{y} \, \phi_{\D} (\tup{x},\tup{y},\tup{a})$ and $\Psi(\tup{b})^{\D}\equiv\existsst \tup{u}\forallst \tup{v} \, \psi_{\D} (\tup{u},\tup{v},\tup{b})$. Then
    \begin{itemize}
      \item[(iii)] $(\Phi (\tup{a}) \land \Psi (\tup{b}))^{\D} :\equiv \existsst \tup{x},\tup{u} \forallst \tup{y},\tup{v} \, \big(\phi_{\D} (\tup{x},\tup{y},\tup{a}) \land \psi_{\D} (\tup{u},\tup{v},\tup{b})\big),$
      \item[(iv)] $(\Phi (\tup{a}) \lor \Psi (\tup{b}))^{\D}  :\equiv \existsst \tup{x},\tup{u} \forallst \tup{y},\tup{v} \, \big(\phi_{\D} (\tup{x},\tup{y},\tup{a}) \lor \psi_{\D} (\tup{u},\tup{v},\tup{b})\big),$

      \item[(v)] $(\Phi(\tup{a}) \to \Psi(\tup{b}))^{\D}  :\equiv   \existsst \tup{U},\tup{Y} \forallst \tup{x},\tup{v} \, \big(\forall \tup{y} \in \tup{Y}[\tup{x},\tup{v}]\, \phi_{\D} (\tup{x},\tup{y}, \tup{a})    \to\psi_{\D} (\tup{U}[\tup{x}],\tup{v}, \tup{b})\big).$
      \end{itemize}
Let $\Phi(z,\tup{a})^{\D}\equiv\existsst \tup{x}\forallst \tup{y} \, \phi_{\D} (\tup{x},\tup{y},z,\tup{a})$, with the free variable $z$ not occuring among the $\tup{a}$. Then
    \begin{itemize}
 \item[(vi)] $(\forall z\Phi(z,\tup{a}))^{\D} :\equiv \existsst \tup{x} \forallst \tup{y} \forall z \, \phi_{\D} (\tup{x},\tup{y},z,\tup{a}),$
 \item[(vii)] $(\exists z\Phi(z,\tup{a}))^{\D} :\equiv \existsst \tup{x} \forallst \tup{y} \exists z \forall \tup{y'} \in \tup{y}\, \phi_{\D} (\tup{x},\tup{y'},z,\tup{a}),$
 \item[(viii)] $(\forallst z\Phi(z,\tup{a}))^{\D} :\equiv \existsst \tup{X} \forallst z,\tup{y}  \, \phi_{\D} (\tup{X}[z],\tup{y},z,\tup{a}),$
 \item[(ix)] $(\existsst z\Phi(z,\tup{a}))^{\D} :\equiv \existsst \tup{x},z \, \forallst \tup{y} \, \exists z' \in z \, \forall \tup y' \in \tup{y}\, \phi_{\D} (\tup{x},\tup{y'},z',\tup{a}).$
\end{itemize}
\end{dfn}
\begin{dfn}
We say that a formula $\Phi$ is a $\forallst$-formula if $\Phi \equiv \forallst \tup{x}\,  \phi (\tup{x})$, with $\phi (\tup{x})$ internal.
\end{dfn}
\begin{lemma}\label{le:forallst-formulas}
Let $\Phi$ be a $\forallst$-formula. Then $\Phi^{\D}\equiv \Phi$.
\end{lemma}
\begin{proof}
By induction on the structure of $\Phi$.
\end{proof}

Notice that because of the clause for $\existsst z$ the interpretation is not idempotent. Similarly to what is the case for Herbrand realizability it will be crucial that realizers are upwards closed:

\begin{lemma} Let $\Phi (\tup{a})$ be a formula in the language of $\ehaststar$ with interpretation $\existsst \tup{x}\forallst \tup{y} \, \phi_{\D} (\tup{x},\tup{y},\tup{a})$. Then the formula $\phi_{\D} (\tup{x},\tup{y},\tup{a})$ is provably upwards closed in $\tup{x}$, i.e.,
\[
          \ehastar \vdash \phi_{\D} (\tup{x},\tup{y},\tup{a}) \land \tup{x} \preceq \tup{x}' \to  \phi_{\D} (\tup{x}',\tup{y},\tup{a}).
\]

\end{lemma}
\begin{proof}
By induction on the structure of $\Phi (\tup{a})$, using Lemma~\ref{le:herbrand:new_application} in the clauses for $\to$ and $\forallst$.
\end{proof}

%Before proving the soundness of the $\D$-interpretation
The $\D$-interpretation will allow us to interpret the nonclassical realization principle $\NCR$, and also both $\I$ and $\HAC$. Additionally we will be able to interpret a herbrandized independence of premise principle for formulas of the form $\forallst x \, \phi(x)$, and also a herbrandized form of a generalized Markov's principle:
\begin{enumerate}
\item $\HIP_{\forallst}$:
\[
           \big( \forallst x \, \phi(x) \to\existsst y\Psi(y)\big)\rightarrow \existsst y \, \big(\forallst x \,  \phi (x)\to \exists y' \in y \, \Psi(y')\big),
            \]
           where $\Psi(y)$ is a formula in the language of $\ehaststar$ and $\phi(x)$ is an internal formula. If $\Psi(y)$ is upwards closed in $y$, then this is equivalent to
           \[
        \big( \forallst x \, \phi(x) \to\existsst y\Psi(y)\big)\rightarrow \existsst y \, \big(\forallst x \,  \phi (x)\to \Psi(y)\big).
\]
\item $\HGMP$:
\[
  ( \forallst x \, \phi(x) \to\psi)\to \existsst x \, \big(\forall x' \in x\,  \phi (x')\to\psi\big),
            \]
           where $\phi(x)$ and $\psi$ are internal formulas in the language of $\ehaststar$. If $\phi(x)$ is downwards closed in $x$, then this is equivalent to
           \[
       ( \forallst x \, \phi(x) \to\psi)\to \existsst x (  \phi (x)\to\psi).
\]
The latter gives us a form of Markov's principle by taking $\psi \equiv 0=_0 1$ and $\phi (x) \equiv \lnot \phi_0 (x)$ (with $\phi_0(x)$ internal and quantifier-free), whence the name.
\end{enumerate}

\begin{thm}[Soundness of the $\D$-interpretation] \label{soundnessDst}
Let $\Phi (\tup{a})$ be a formula of $\ehaststar$ and let $\Delta_{\intern}$ be a set of internal sentences.
If
\[
    \ehaststar + \I + \NCR + \HAC + \HGMP
    %+ \LLPO
    +  \HIP_{\forallst} + \Delta_{\intern} \vdash \Phi (\tup{a})
\]
and $\Phi(\tup{a})^{\D}\equiv\existsst \tup{x}\forallst \tup{y} \, \phi_{\D} (\tup{x},\tup{y},\tup{a})$,
then from the proof we can extract closed terms $\tup{t}$ in $\Tstar$ such that
\[
    \ehastar +\Delta_{\intern} \vdash \forall \tup{y} \, \phi_{\D} (\tup{t},\tup{y},\tup{a}).
\]
\end{thm}
\begin{proof} As in the proof of the soundness of the Dialectica interpretation we proceed by induction on the length of the derivation.

\begin{enumerate}
\item
We will first consider the logical axioms and rules:
\begin{enumerate}

\item $A \to A \land A$:

With $A^{\D}\equiv \existsst \tup{x}\forallst \tup{y} \, \phi (\tup{x},\tup{y},\tup{a})$ we have
\begin{displaymath}
\begin{array}{c}
                  (A \to A \land A)^{\D}   \equiv  \existsst \tup{X}',\tup{X}'',\tup{Y} \forallst \tup{x},\tup{y}',\tup{y}'' \\
               \Big( \forall \tup{z} \in \tup{Y} [\tup{x},\tup{y}', \tup{y}'']  
\phi (\tup{x},  \tup{z},\tup{a}  )
\to
\phi(\tup{X}'[\tup{x}],\tup{y}',\tup{a}) \land \phi(\tup{X}''[\tup{x}],\tup{y}'',\tup{a}) \Big),
\end{array}
\end{displaymath}
and we can take
\begin{eqnarray*}
     \tup{X}'   &:=&  \Lambda \, \tup{x} \, . \, \tup{x}, \\
     \tup{X}''  &:=&  \Lambda \, \tup{x} \, . \, \tup{x}, \\
     \tup{Y}  &:=&   \Lambda \, \tup{x} ,\tup{y}' ,\tup{y}'' \, . \, \tup{\langle \tup{y}', \tup{y}'' \rangle}.
\end{eqnarray*}
%such that
%\begin{eqnarray*}
%     \tup{X}' [\tup{x}]  &=&  \tup{x} \\
%     \tup{X}'' [\tup{x}]  &=&  \tup{x} \\
%     \tup{Y} [\tup{x},\tup{y}',\tup{y}'']  &=& \langle \tup{y}', \tup{y}'' \rangle.
%\end{eqnarray*}

\item $A \lor A \to A$:

With $A^{\D}\equiv \existsst \tup{x}\forallst \tup{y} \, \phi (\tup{x},\tup{y},\tup{a})$ we have
\begin{displaymath}
\begin{array}{c}
(A \lor A \to A)^{\D}   \equiv 
 \existsst \tup{X}'',\tup{Y},\tup{Y}' \forallst \tup{x},\tup{x}',\tup{y}'' \\
\Big( (\forall \tup{z} \in \tup{Y} [\tup{x},\tup{x}', \tup{y}'']  \phi (\tup{x},  \tup{z},\tup{a}  )     \lor \
\forall \tup{z'} \in \tup{Y}' [\tup{x},\tup{x}', \tup{y}'']
\phi (\tup{x}',  \tup{z'},\tup{a}  )
 ) \\
  \to
 \phi (\tup{X}''[\tup{x},\tup{x}'],  \tup{y}'',\tup{a}  )
  \Big),
\end{array}
\end{displaymath}
and we can take
\begin{eqnarray*}
     \tup{X}''  &:=&  \Lambda \, \tup{x},\tup{x}' \, . \,  \tup{x} * \tup{x}', \\
     \tup{Y}  &:=& \Lambda \, \tup{x},\tup{y},\tup{y}'' \, . \, \tup{\langle \tup{y}'' \rangle}, \\
     \tup{Y}'  &:=& \Lambda \, \tup{x},\tup{y},\tup{y}'' \, . \, \tup{\langle \tup{y}'' \rangle}.
\end{eqnarray*}
%\begin{eqnarray*}
%     \tup{X}'' [\tup{x},\tup{x}']  &=&  \tup{x} * \tup{x}' \\
%     \tup{Y} [\tup{x},\tup{y},\tup{y}'']  &=& \langle \tup{y}'' \rangle \\
%     \tup{Y}' [\tup{x},\tup{y},\tup{y}'']  &=& \langle \tup{y}'' \rangle.
%\end{eqnarray*}

\item $A \to A \lor B$:

With $A^{\D}\equiv \existsst \tup{x}\forallst \tup{y} \, \phi (\tup{x},\tup{y},\tup{a})$ and $B^{\D}\equiv \existsst \tup{u}\forallst \tup{v} \, \psi (\tup{u},\tup{v},\tup{b})$ we have
\begin{displaymath}
\begin{array}{c}
                  (A \to A \lor B)^{\D}   \equiv  \existsst \tup{X}',\tup{U},\tup{Y} \forallst \tup{x},\tup{y}',\tup{v} \\
               \Big( \forall \tup{z} \in \tup{Y} [\tup{x},\tup{y}', \tup{v}]
\phi (\tup{x},  \tup{z} ,\tup{a}  ) \to
\phi(\tup{X}'[\tup{x}],\tup{y}',\tup{a}) \lor \psi(\tup{U}[\tup{x}],\tup{v},\tup{b}) \Big),
\end{array}
\end{displaymath}
and we can take
\begin{eqnarray*}
     \tup{X}'  &:=& \Lambda \, \tup{x} \, . \,  \tup{x}, \\
     \tup{Y}   &:=& \Lambda \, \tup{x},\tup{y}',\tup{v} \, . \, \tup{\langle \tup{y}' \rangle}, \\
     \tup{U}   &:=& \Lambda \, \tup{x} \, . \,  \tup{\mathcal{O}}.
\end{eqnarray*}
%\begin{eqnarray*}
%     \tup{X}' [\tup{x}]  &=&  \tup{x} \\
%     \tup{Y} [\tup{x},\tup{y}',\tup{v}]  &=& \langle \tup{y}' \rangle \\
%     \tup{U} [\tup{x}]  &=&  \tup{\mathcal{O}}.
%\end{eqnarray*}

\item $A \land B \to A$:

With $A^{\D}\equiv \existsst \tup{x}\forallst \tup{y} \, \phi (\tup{x},\tup{y},\tup{a})$ and $B^{\D}\equiv \existsst \tup{u}\forallst \tup{v} \, \psi (\tup{u},\tup{v},\tup{b})$ we have
\begin{displaymath}
\begin{array}{c}
                  (A \land B \to A)^{\D}   \equiv  \existsst \tup{X}',\tup{Y},\tup{V} \forallst \tup{x},\tup{u},\tup{y}' \\ 
               \Big( \forall \tup{z} \in \tup{Y} [\tup{x},\tup{u}, \tup{y}'] \, \forall \tup{t} \in \tup{V} [\tup{x},\tup{u}, \tup{y}']
\big(\phi (\tup{x},  \tup{z},\tup{a}  )  \land  \, \psi (\tup{u},  \tup{t},\tup{b}  )\big)   \to
\phi(\tup{X}'[\tup{x}, \tup{u}],\tup{y}',\tup{a})  \Big),
\end{array}
\end{displaymath}
and we can take
\begin{eqnarray*}
     \tup{X}'  &:=&  \Lambda \, \tup{x},\tup{u} \, . \, \tup{x}, \\
     \tup{Y}  &:=& \Lambda \, \tup{x},\tup{u},\tup{y}' \, . \, \tup{\langle \tup{y}' \rangle}, \\
     \tup{V}  &:=&  \Lambda \, \tup{x},\tup{u},\tup{y}' \, . \, \tup{\mathcal{O}}.
\end{eqnarray*}
%\begin{eqnarray*}
%     \tup{X}' [\tup{x}, \tup{u}]  &=&  \tup{x} \\
%     \tup{Y} [\tup{x},\tup{u},\tup{y}']  &=& \langle \tup{y}' \rangle \\
%     \tup{V} [\tup{x},\tup{u},\tup{y}']  &=&  \tup{\mathcal{O}}.
%\end{eqnarray*}

\item $A \lor B \to B \lor A$:

With $A^{\D}\equiv \existsst \tup{x}\forallst \tup{y} \, \phi (\tup{x},\tup{y},\tup{a})$ and $B^{\D}\equiv \existsst \tup{u}\forallst \tup{v} \, \psi (\tup{u},\tup{v},\tup{b})$ we have
\begin{displaymath}
\begin{array}{c}
                  (A \lor B \to B \lor A)^{\D}   \equiv  \existsst \tup{U}',\tup{X}',\tup{Y},\tup{V} \forallst \tup{x},\tup{u},\tup{v}',\tup{y}' \\
               \Big( \forall \tup{z} \in \tup{Y} [\tup{x},\tup{u}, \tup{v}',\tup{y}'] \, \forall \tup{t} \in \tup{V} [\tup{x},\tup{u},\tup{v}', \tup{y}'] 
\big(\phi (\tup{x},  \tup{z},\tup{a}  )   \lor  \, \psi (\tup{u},  \tup{t}, \tup{b}  )\big)  \\ \to
\big(  \psi(\tup{U}'[\tup{x}, \tup{u}],\tup{v}',\tup{b}) \lor  \phi(\tup{X}'[\tup{x}, \tup{u}],\tup{y}',\tup{a}) \big) \Big),
\end{array}
\end{displaymath}
and we can take
\begin{eqnarray*}
     \tup{U}'  &:=&  \Lambda \, \tup{x},\tup{u} \, . \, \tup{u}, \\
     \tup{X}'  &:=&  \Lambda \, \tup{x},\tup{u} \, . \, \tup{x}, \\
     \tup{Y}  &:=& \Lambda \, \tup{x},\tup{u},\tup{v}',\tup{y}' \, . \, \tup{\langle \tup{y}' \rangle}, \\
     \tup{V}  &:=& \Lambda \, \tup{x},\tup{u},\tup{v}',\tup{y}' \, . \, \tup{\langle \tup{v}' \rangle}.
\end{eqnarray*}
%\begin{eqnarray*}
%     \tup{U}' [\tup{x}, \tup{u}]  &=&  \tup{u} \\
%     \tup{X}' [\tup{x}, \tup{u}]  &=&  \tup{x} \\
%     \tup{Y} [\tup{x},\tup{u},\tup{v}',\tup{y}']  &=& \langle \tup{y}' \rangle \\
%     \tup{V} [\tup{x},\tup{u},\tup{v}',\tup{y}']  &=& \langle \tup{v}' \rangle.
%\end{eqnarray*}

\item $A \land B \to B \land A$:

We can take the same terms as for $A \lor B \to B \lor A$, since $\lor$ and $\land$ are handled in the same way by the interpretation.

\item $\bot \to A$:

With $A^{\D}\equiv \existsst \tup{x}\forallst \tup{y} \, \phi (\tup{x},\tup{y},\tup{a})$ we have
\[
                  (\bot \to A)^{\D}   \equiv  \existsst \tup{X} \forallst \tup{y} \, \big(\bot \to \phi  (\tup{X},\tup{y},\tup{a}) \big),
\]
and we can take $\tup{X} :=  \tup{\mathcal{O}}$.

\item $\forall z A  \to A[t/z]$:
%\item $\forall z A (z) \to A(t)$:

With $A^{\D}\equiv \existsst \tup{x}\forallst \tup{y} \, \phi (\tup{x},\tup{y},z,\tup{a})$
%$A(z)^{\D}\equiv \existsst \tup{x}\forallst \tup{y} \, \phi (\tup{x},\tup{y},z,\tup{a})$
 we have
\begin{displaymath}
\begin{array}{c}
(\forall z A  \to A[t/z])^{\D}   \equiv 
 \existsst \tup{X}',\tup{Y} \forallst \tup{x},\tup{y}' \\
\Big( \forall \tup{w} \in \tup{Y} [\tup{x},\tup{y}'] \forall z\,
\phi (\tup{x},  \tup{w}, z ,\tup{a}  )   \to
 \phi (\tup{X}'[\tup{x}],  \tup{y}',t,\tup{a}  )
  \Big),
\end{array}
\end{displaymath}
%\begin{eqnarray*}
%(\forall z A (z) \to A(t))^{\D}  & \equiv &
% \existsst \tup{X}',\tup{Y} \forallst \tup{x},\tup{y}'
%\Big( \forall \tup{i} \leq |\tup{Y} [\tup{x},\tup{y}']| \forall z\,
%\phi (\tup{x},  (\tup{Y} [\tup{x}, \tup{y}'])_{\tup{i}},z,\tup{a}  )   \\ && \quad  \to
% \phi (\tup{X}'[\tup{x}],  \tup{y}',t,\tup{a}  )
%  \Big),
%\end{eqnarray*}
and we can take
\begin{eqnarray*}
     \tup{X}'   &:=& \Lambda \, \tup{x} \, . \, \tup{x}, \\
     \tup{Y}   &:=& \Lambda \, \tup{x},\tup{y}' \, . \, \tup{\langle \tup{y}' \rangle}.
\end{eqnarray*}
%\begin{eqnarray*}
%     \tup{X}' [\tup{x}]  &=&  \tup{x} \\
%     \tup{Y} [\tup{x},\tup{y}']  &=& \langle \tup{y}' \rangle.
%\end{eqnarray*}

\item $A[t/z] \to \exists z A $:
%\item $A(t) \to \exists z A (z)$:

With $A^{\D}\equiv \existsst \tup{x}\forallst \tup{y} \, \phi (\tup{x},\tup{y},z,\tup{a})$
%$A(z)^{\D}\equiv \existsst \tup{x}\forallst \tup{y} \, \phi (\tup{x},\tup{y},z,\tup{a})$
 we have
\begin{displaymath}
\begin{array}{c}
(A[t/z] \to \exists z A)^{\D}   \equiv 
 \existsst \tup{X}',\tup{Y} \forallst \tup{x},\tup{y}' \\
\Big( \forall \tup{w} \in \tup{Y} [\tup{x},\tup{y}'] \,
\phi (\tup{x},  \tup{w} ,t,\tup{a}  )   \to
 \exists z\forall \tup{w} \in \tup{y}'\, \phi (\tup{X}'[\tup{x}],  \tup{w},z,\tup{a}  )
  \Big),
\end{array}
\end{displaymath}
%\begin{eqnarray*}
%(A(t) \to \exists z A (z))^{\D}  & \equiv &
% \existsst \tup{X}',\tup{Y} \forallst \tup{x},\tup{y}'
%\Big( \forall \tup{i} \leq |\tup{Y} [\tup{x},\tup{y}']| \,
%\phi (\tup{x},  (\tup{Y} [\tup{x}, \tup{y}'])_{\tup{i}},t,\tup{a}  )   \\ && \quad  \to
% \exists z\forall \tup{k} \leq |\tup{y}|\, \phi (\tup{X}'[\tup{x}],  \tup{y}'_{\tup{k}},z,\tup{a}  )
%  \Big),
%\end{eqnarray*}
and we can take
\begin{eqnarray*}
     \tup{X}'  &:=& \Lambda \, \tup{x} \, . \, \tup{x}, \\
     \tup{Y}   &:=& \Lambda \, \tup{x},\tup{y}' \, . \, \tup{y}' .
\end{eqnarray*}
%\begin{eqnarray*}
%     \tup{X}' [\tup{x}]  &=&  \tup{x} \\
%     \tup{Y} [\tup{x},\tup{y}']  &=&  \tup{y}' .
%\end{eqnarray*}

\item The modus ponens rule:

We assume that we have terms $\tup{t}_1$ and $\tup{T}_2,\tup{T}_3 $ realizing the interpretations of respectively $A$ and $A\to B$, and we wish to obtain terms $\tup{T}_4$ realizing the interpretation of $B$. So
%let $\tup{t}_1,\tup{T}_2,\tup{T}_3 $ be such that
with $A^{\D}\equiv \existsst \tup{x}\forallst \tup{y} \, \phi (\tup{x},\tup{y},\tup{a})$ and $B^{\D}\equiv \existsst \tup{u}\forallst \tup{v} \, \psi (\tup{u},\tup{v},\tup{b})$ we have
\[
         \ehastar+\Delta_{\intern} \vdash \forall \tup{y} \, \phi (\tup{t}_1,\tup{y},\tup{a})
\]
and
\[
           \ehastar+\Delta_{\intern} \vdash \forall \tup{x},\tup{v} \, \big(\forall \tup{z}\in \tup{T}_3[\tup{x},\tup{v}]\, \phi (\tup{x}, \tup{z}, \tup{a})   \to\psi (\tup{T}_2[\tup{x}],\tup{v}, \tup{b})\big).
\]
With $\tup{T}_4 := \tup{T}_2 [\tup{t}_1]$ we get
\[
        \ehastar+\Delta_{\intern} \vdash \forall \tup{v} \, \psi (\tup{T}_4,\tup{v},\tup{b}),
\]
as desired.

\item The syllogism rule:

We assume that we have terms $\tup{T}_1,\tup{T}_2$ and $\tup{T}_3,\tup{T}_4 $ realizing the interpretations of respectively $A\to B$ and $B\to C$, and we wish to obtain terms $\tup{T}_5,\tup{T}_6$ realizing the interpretation of $A\to C$. So
%let $\tup{t}_1,\tup{T}_2,\tup{T}_3 $ be such that
with $A^{\D}\equiv \existsst \tup{x}\forallst \tup{y} \, \phi (\tup{x},\tup{y},\tup{a})$, $B^{\D}\equiv \existsst \tup{u}\forallst \tup{v} \, \psi (\tup{u},\tup{v},\tup{b})$, and $C^{\D}\equiv \existsst \tup{w}\forallst \tup{z} \, \chi (\tup{w},\tup{z},\tup{c})$ we have
\begin{equation}\label{eq:syllogism_1}
           \ehastar+\Delta_{\intern} \vdash \forall \tup{x},\tup{v} \, \big(\forall \tup{i}\in \tup{T}_2[\tup{x},\tup{v}]\, \phi (\tup{x},\tup{i}, \tup{a})   \to\psi (\tup{T}_1[\tup{x}],\tup{v}, \tup{b})\big)
\end{equation}
and
\begin{equation}\label{eq:syllogism_2}
           \ehastar+\Delta_{\intern} \vdash \forall \tup{u},\tup{z} \, \big(\forall \tup{j}\in \tup{T}_4[\tup{u},\tup{z}|\, \psi (\tup{u},\tup{j}, \tup{b})   \to\chi (\tup{T}_3[\tup{u}],\tup{z}, \tup{c})\big),
\end{equation}
and we wish to obtain $\tup{T}_5,\tup{T}_6$ such that
\begin{equation*}
           \ehastar+\Delta_{\intern} \vdash \forall \tup{x},\tup{z} \, \big(\forall \tup{k} \in \tup{T}_6[\tup{x},\tup{z}]\, \phi (\tup{x},\tup{k}, \tup{a})   \to\chi (\tup{T}_5[\tup{x}],\tup{z}, \tup{c})\big).
\end{equation*}
To do this we let $\tup{T}_5  := \Lambda \, \tup{x} \, . \, \tup{T}_3 [\tup{T}_1[\tup{x}]]$,  $\tup{T}_7  := \Lambda \, \tup{x},\tup{z} \, . \, \tup{T}_4 [\tup{T}_1[\tup{x}],\tup{z}]$ and apply~(\ref{eq:syllogism_2}) with $\tup{u} = \tup{T}_1 [\tup{x}] $, such that
\begin{equation*}
           \ehastar+\Delta_{\intern} \vdash \forall \tup{x},\tup{z} \, \big(\forall \tup{j} \in \tup{T}_7[\tup{x},\tup{z}]\, \psi (\tup{T}_1[\tup{x}],\tup{j}, \tup{b})   \to\chi (\tup{T}_5[\tup{x}],\tup{z}, \tup{c})\big).
\end{equation*}
As a special case of~(\ref{eq:syllogism_1}) we get
\begin{displaymath}
\begin{array}{c}
           \ehastar+\Delta_{\intern}  \vdash   \forall \tup{x},\tup{z} \forall \tup{j} \in \tup{T}_7 [\tup{x},\tup{z}] \,  \\
           \big(\forall \tup{i} \in \tup{T}_2[\tup{x},\tup{j}] \, \phi (\tup{x},\tup{i}, \tup{a})  \to\psi (\tup{T}_1[\tup{x}],   \tup{j}   , \tup{b})\big),
\end{array}
\end{displaymath}
and thus
\begin{displaymath}
\begin{array}{c}
           \ehastar +\Delta_{\intern}  \vdash   \forall \tup{x},\tup{z} \\
       \big(  \forall \tup{j} \in \tup{T}_7 [\tup{x},\tup{z}] \, \forall \tup{i} \in \tup{T}_2[\tup{x},\tup{j} ] \,   \phi (\tup{x},\tup{i}, \tup{a})  \to  \chi (\tup{T}_5[\tup{x}],\tup{z}, \tup{c})  \big).
\end{array}
\end{displaymath}
Hence it is enough to construct terms $\tup{T}_6$ such that
\[
      \tup{T}_6[ \tup{x},\tup{z}] =_{\tup{\sigma}} \tup{T}_2 [\tup{x} , \tup{T}_7[\tup{x},\tup{z}]_{\tup{0}} ] *_{\tup{\sigma}}\ldots *_{\tup{\sigma}} \tup{T}_2 [\tup{x} , \tup{T}_7[\tup{x},\tup{z}]_{ | \tup{T}_7 [\tup{x} , \tup{z}] | - \tup{1}} ],
\]
where the concatenation is to include $\tup{T}_2 [\tup{x} , \tup{T}_7[\tup{x},\tup{z}]_{\tup{j}} ]$ for each $\tup{j} <  | \tup{T}_7 [\tup{x} , \tup{z}] |$, which one can easily do using terms from $\Tstar$.

\item The importation and exportation rules:

With $A^{\D}\equiv \existsst \tup{x}\forallst \tup{y} \, \phi (\tup{x},\tup{y},\tup{a})$, $B^{\D}\equiv \existsst \tup{u}\forallst \tup{v} \, \psi (\tup{u},\tup{v},\tup{b})$, and $C^{\D}\equiv \existsst \tup{w}\forallst \tup{z} \, \chi (\tup{w},\tup{z},\tup{c})$ we get
\begin{displaymath}
\begin{array}{c}
                  (A \land B \to C)^{\D}   \equiv  \existsst \tup{W},\tup{Y},\tup{V} \forallst \tup{x},\tup{u},\tup{z} \\ 
               \Big( \forall \tup{i} \in \tup{Y} [\tup{x},\tup{u}, \tup{z}] \, \forall \tup{j} \in \tup{V} [\tup{x},\tup{u}, \tup{z}]
\big(\phi (\tup{x},  \tup{i} ,\tup{a}  )   \land  \, \psi (\tup{u},  \tup{j} ,\tup{b}  )\big)   \to
\chi(\tup{W}[\tup{x}, \tup{u}],\tup{z},\tup{c})  \Big)
\end{array}
\end{displaymath}
and
\begin{displaymath}
\begin{array}{c}
                  (A \to (B \to C))^{\D}  \equiv \existsst \tup{W},\tup{Y},\tup{V} \forallst \tup{x},\tup{u},\tup{z} \\
               \Big( \forall \tup{i} \in \tup{Y} [\tup{x},\tup{u}, \tup{z}] \,
               \phi (\tup{x},  \tup{i} ,\tup{a}  )  \to \big(
               \forall \tup{j} \in \tup{V} [\tup{x},\tup{u}, \tup{z}]
 \, \psi (\tup{u},  \tup{j},\tup{b}  )  \to
\chi(\tup{W}[\tup{x}, \tup{u}],\tup{z},\tup{c})  \big) \Big),
\end{array}
\end{displaymath}
so that the same terms realize the interpretations of $A \land B \to C$ and $A \to (B \to C)$.

\item The expansion rule:

We assume that we have terms $\tup{T}_1,\tup{T}_2 $ realizing the interpretation of $A\to B$, and we wish to obtain terms $\tup{T}_3,\tup{T}_4,\tup{T}_5,\tup{T}_6$ realizing the interpretation of $C \lor A\to C \lor B$. So with $A^{\D}\equiv \existsst \tup{x}\forallst \tup{y} \, \phi (\tup{x},\tup{y},\tup{a})$, $B^{\D}\equiv \existsst \tup{u}\forallst \tup{v} \, \psi (\tup{u},\tup{v},\tup{b})$, and $C^{\D}\equiv \existsst \tup{w}\forallst \tup{z} \, \chi (\tup{w},\tup{z},\tup{c})$ we have
\[
           \ehastar+\Delta_{\intern} \vdash \forall \tup{x},\tup{v} \, \big(\forall \tup{i}\in \tup{T}_2[\tup{x},\tup{v}]\, \phi (\tup{x},\tup{i}, \tup{a})   \to\psi (\tup{T}_1[\tup{x}],\tup{v}, \tup{b})\big),
\]
and we want $\tup{T}_3,\tup{T}_4,\tup{T}_5,\tup{T}_6$ such that
\begin{displaymath}
\begin{array}{c}
                  \ehastar+\Delta_{\intern}     \vdash  \\ \forall \tup{w},\tup{x},\tup{z}',\tup{v} 
               \Big( \forall \tup{j} \in \tup{T}_5 [\tup{w},\tup{x},\tup{z}',\tup{v}] \, \forall \tup{i} \in \tup{T}_6 [\tup{w},\tup{x},\tup{z}',\tup{v}]  
\big(\chi (\tup{w},  \tup{j},\tup{c}  )   \, \lor  \, \phi (\tup{x},  \tup{i},\tup{a}  )\big)  \\   \to
 \chi(\tup{T}_3[\tup{w}, \tup{x}],\tup{z}',\tup{c})   \lor  \psi(\tup{T}_4[\tup{w}, \tup{x}],\tup{v},\tup{b})   \Big).
\end{array}
\end{displaymath}
Thus we can take
\begin{eqnarray*}
     \tup{T}_3   &:=& \Lambda \, \tup{w},\tup{x} \, . \, \tup{w}, \\
     \tup{T}_4   &:=& \Lambda \, \tup{w},\tup{x} \, . \, \tup{T}_1 [\tup{x}], \\
     \tup{T}_5   &:=& \Lambda \, \tup{w},\tup{x},\tup{z}',\tup{v} \, . \, \tup{\langle \tup{z}' \rangle}, \\
     \tup{T}_6   &:=& \Lambda \, \tup{w},\tup{x},\tup{z}',\tup{v} \, . \, \tup{T}_2 [\tup{x},\tup{v}].
\end{eqnarray*}

\item The quantifier rules:

\begin{enumerate}
\item We assume that we have terms $\tup{T}_1,\tup{T}_2 $ realizing the interpretation of $B\to A$, and we want terms $\tup{T}_3,\tup{T}_4$ realizing the interpretation of $B \to \forall z A$, where $z$ is not among the free variables of $B$. Thus with $A^{\D}\equiv \existsst \tup{x}\forallst \tup{y} \, \phi (\tup{x},\tup{y},z,\tup{a})$ and $B^{\D}\equiv \existsst \tup{u}\forallst \tup{v} \, \psi (\tup{u},\tup{v},\tup{b})$ we have
\[
           \ehastar+\Delta_{\intern} \vdash \forall \tup{u},\tup{y} \, \big(\forall \tup{i}\in \tup{T}_2[\tup{u},\tup{y}]\, \psi (\tup{u},\tup{i}, \tup{b})   \to\phi (\tup{T}_1[\tup{u}],\tup{y}, z,\tup{a})\big),
\]
and we want
\[
           \ehastar+\Delta_{\intern} \vdash \forall \tup{u},\tup{y} \, \big(\forall \tup{i} \in \tup{T}_4[\tup{u},\tup{y}]\, \psi (\tup{u},\tup{i}, \tup{b})   \to \forall z \, \phi (\tup{T}_3[\tup{u}],\tup{y}, z,\tup{a})\big).
\]
Hence we may take $\tup{T}_3 := \tup{T}_1$ and $\tup{T}_4 := \tup{T}_2$.

\item We have terms $\tup{T}_1,\tup{T}_2 $ realizing the interpretation of $A\to B$, and we want terms $\tup{T}_3,\tup{T}_4$ realizing the interpretation of $\exists z A \to B$, where $z$ is not among the free variables of $B$. Thus with $A^{\D}\equiv \existsst \tup{x}\forallst \tup{y} \, \phi (\tup{x},\tup{y},z,\tup{a})$ and $B^{\D}\equiv \existsst \tup{u}\forallst \tup{v} \, \psi (\tup{u},\tup{v},\tup{b})$ we have
\[
           \ehastar+\Delta_{\intern} \vdash \forall \tup{x},\tup{v} \, \big(\forall \tup{i} \in \tup{T}_2[\tup{x},\tup{v}]\, \phi (\tup{x},\tup{i},z, \tup{a})   \to\psi (\tup{T}_1[\tup{x}],\tup{v}, \tup{b})\big),
\]
and we want
\[
           \ehastar+\Delta_{\intern}     \vdash      \forall \tup{x},\tup{v} \, \big(
           \forall \tup{k} \in \tup{T}_4  [\tup{x},\tup{v}]   \, \exists z
           \forall \tup{i}\in \tup{k}        \phi (\tup{x}, \tup{i}, z,\tup{a})   \to\psi (\tup{T}_3[\tup{x}],\tup{v}, \tup{b})\big).\]
So we can take $\tup{T}_3 := \tup{T}_1$ and $\tup{T}_4 :=   \Lambda \, \tup{x},\tup{v} \, . \, \tup{\langle \, \tup{T}_2 [\tup{x}, \tup{v}] \, \rangle}$.

\end{enumerate}

\end{enumerate}
\item The nonlogical axioms of $\ehastar$: These axioms are all internal, and so their $\D$-interpretations are all realized by the empty tuple of terms.
\item The defining axioms $\EQ$ of the external quantifiers:
\begin{enumerate}
\item $\forallst x \, \Phi (x) \leftrightarrow \forall x ( \st (x) \to \Phi (x))$:

We treat first $\forallst x \, \Phi (x) \to \forall x ( \st (x) \to \Phi (x))$. With $\Phi^{\D}\equiv \existsst \tup{u}\forallst \tup{v} \, \phi (\tup{u},\tup{v},x,\tup{a})$ we get
\[
           (\forallst x \, \Phi (x) )^{\D} \equiv \existsst \tup{U} \forallst x,\tup{v}\, \phi(\tup{U}[x], \tup{v},x,\tup{a})
\]
and
\[
           \big(\forall x\, (\st(x) \to \Phi (x)) \big)^{\D} \equiv
           \existsst \tup{U}' \forallst z,\tup{v}'
           \forall x \, \big( x \in z \to
           \phi(\tup{U}'[z], \tup{v}',x,\tup{a})\big),
\]
and thus  $\big(\forallst x \, \Phi (x)  \to \forall x\, (\st(x) \to \Phi (x)) \big)^{\D}$ is
\begin{displaymath}
\begin{array}{c}
%           \big(\forallst x \, \Phi (x)  \to \forall x\, (\st(x) \to \Phi (x)) \big)^{\D} &  \equiv &
           \existsst \tup{\tilde{U}}' , X, \tup{V} \, \forallst \tup{U},z,\tup{v}' \\
%          \\ &&
            \Big(
           \forall i \in X[\tup{U}, z,\tup{v}']  \,  \forall \tup{j} \in \tup{V} [\tup{U}, z, \tup{v}']
            \phi\big( \tup{U}[i] , \tup{j} , i , \tup{a}\big) \to \forall x \big( \, x \in z \to \phi ( \tup{\tilde{U}}' [\tup{U} , z] , \tup{v}' , x , \tup{a}     )\big)
           \Big).
\end{array}
\end{displaymath}
Let now
\begin{eqnarray*}
     X   &:=& \Lambda \, \tup{U},z,\tup{v}' \, . \, z, \\
     \tup{V}   &:=& \Lambda \,  \tup{U},z,\tup{v}' \, . \, \tup{\langle \tup{v}' \rangle}, \\
     \tup{\tilde{U}}'   &:=& \Lambda \,  \tup{U},z  \, . \, \tup{U}[z_0] * \ldots * \tup{U}[z_{|z|-1}]
\end{eqnarray*}
(with
\[
      \tup{U}[z_0] * \ldots * \tup{U}[z_{|z|-1}] := U_0[z_0] * \ldots * U_0[z_{|z|-1}], \ldots,                                                                               U_{n-1}[z_0] * \ldots * U_n[z_{|z|-1}]
\]
for $\tup{U}=U_0,\ldots,U_{n-1}$). 

Next we treat $\forall x ( \st (x) \to \Phi (x))  \to \forallst x \, \Phi (x)$.
Similarly to the case above we get that $\big(\forall x ( \st (x) \to \Phi (x))  \to \forallst x \, \Phi (x)\big)^{\D}$ is
\begin{eqnarray*}
          & \existsst \tup{\tilde{U}}, Z, \tup{V}' \, \forallst \tup{U}' ,x,\tup{v} \, \\ & 
            \Big(  \forall k  \in Z[\tup{U}', x,\tup{v}] \,  \forall \tup{j} \in \tup{V}' [\tup{U}', x, \tup{v}]                         \forall x' \in k \, \phi ( \tup{U}' [k] ,   \tup{j}  , x' , \tup{a})  \to  \phi ( \tup{\tilde{U}} [\tup{U}' , x] , \tup{v} , x , \tup{a}     )
           \Big).
\end{eqnarray*}
Hence we can take
\begin{eqnarray*}
     Z   &:=& \Lambda \, \tup{U}',x,\tup{v} \, . \, \langle \langle x \rangle \rangle, \\
     \tup{V}'   &:=& \Lambda \,  \tup{U}' ,x,\tup{v} \, . \, \tup{\langle \tup{v} \rangle}, \\
     \tup{\tilde{U}}   &:=& \Lambda \,  \tup{U}',x  \, . \, \tup{U}' [  \langle x\rangle ].
\end{eqnarray*}

\item $\existsst x \, \Phi (x) \leftrightarrow \exists x ( \st (x) \land \Phi (x))$:

First we treat $\existsst x \, \Phi (x) \to \exists x ( \st (x) \land \Phi (x))$. With $\Phi^{\D}\equiv \existsst \tup{u}\forallst \tup{v} \, \phi (\tup{u},\tup{v},x,\tup{a})$ we get
\[
           (\existsst x \, \Phi (x) )^{\D} \equiv \existsst \tup{u},x \forallst \tup{v}\exists i \in x \forall \tup{j} \in \tup{v}\, \phi\big(\tup{u}, \tup{j},i,\tup{a}\big)
\]
and
\[
           \big(\exists x\, (\st(x) \land \Phi (x)) \big)^{\D} \equiv
           \existsst z,\tup{u}' \forallst \tup{v}'
           \exists x' \forall \tup{j} \in \tup{v}' \, \big(x' \in z \, \land \,
           \phi(\tup{u}', \tup{j},x',\tup{a})\big),
\]
and thus $\big(\existsst x \, \Phi (x) \to \exists x ( \st (x) \land \Phi (x))\big)^{\D}$ is
\begin{eqnarray*}
          & \existsst Z,\tup{U}' , \tup{V} \, \forallst \tup{u},x,\tup{v}' \\ &
            \Big(
           \forall \tup{l} \in \tup{V}[\tup{u}, x,\tup{v}']   \,
           \exists k \in x \, \forall \tup{n} \in \tup{l} \, \phi\big( \tup{u}, \tup{n} , k , \tup{a}\big) \to \\ & \exists x'  \,\forall \tup{j} \in \tup{v}' \big(x' \in Z[\tup{u}, x] \, \land \, \phi ( \tup{U}' [\tup{u} , x] , \tup{j} , x' , \tup{a}     )\big)
           \Big).
\end{eqnarray*}
Thus we can take
\begin{eqnarray*}
     Z   &:=& \Lambda \, \tup{u},x \, . \,  x,  \\
     \tup{U}'   &:=& \Lambda \,  \tup{u} ,x \, . \,  \tup{u},  \\
     \tup{V}   &:=& \Lambda \,  \tup{u},x,\tup{v}'  \, . \,   \tup{\langle \tup{v}' \rangle} .
\end{eqnarray*}

Finally we consider $\exists x ( \st (x) \land \Phi (x)) \to \existsst x \, \Phi (x)$. Similarly to the case above we get that $\big(\exists x ( \st (x) \land \Phi (x)) \to \existsst x \, \Phi (x) \big)^{\D}$ is
\begin{eqnarray*}
          & \existsst \tup{U}, X, \tup{V}' \, \forallst z, \tup{u}' ,\tup{v} \, \\ &
            \Big(  \forall \tup{l} \in \tup{V}' [z,\tup{u}', \tup{v}]  \, \exists x' \, \forall \tup{j} \in \tup{l}
                        \big(x' \in z
           \,  \land \, \phi ( \tup{u}'  ,  \tup{j}  , x' , \tup{a} )\big)  \to  \\ & \exists k \in X[z,\tup{u}'] \, \forall \tup{n} \in \tup{v} \, \phi ( \tup{U} [z,\tup{u}' ] , \tup{n} , k , \tup{a}     )
           \Big).
\end{eqnarray*}
Let now
\begin{eqnarray*}
     X   &:=& \Lambda \, z,\tup{u}' \, . \,  z,  \\
     \tup{U}   &:=& \Lambda \,  z,\tup{u}' \, . \,  \tup{u}',  \\
     \tup{V}'   &:=& \Lambda \,  z,\tup{u}',\tup{v}  \, . \,   \tup{\langle \tup{v} \rangle} .
\end{eqnarray*}

\end{enumerate}

\item The schemata $\Tst$:
\begin{enumerate}
\item $\st (x) \, \land \, x=y \to \st (y)$:  We have that $\big(\st (x) \, \land \, x=y \to \st (y)\big)^{\D}$ is
\[
     \existsst V \forallst u \, \big(x \in u \, \land \, x=y \to y \in V[u]\big),
\]
and so we can take $V :=  \Lambda \, u\, . \,  u  $.

\item $\st (t)$ for closed terms $t$ in $\Tstar$: Since
\[
         \big(\st (t)\big)^{\D} \equiv \existsst x \, t \in x,
\]
we can take $x:= \langle t \rangle$.
\item $\st (f^{\sigma \to \tau}) \, \land \, \st (x^{\sigma}) \to \st (f(x))$: We have that $\big(\st (f) \, \land \, \st (x) \to \st (f(x))\big)^{\D}$ is
\begin{eqnarray*}
    &  \existsst W \forallst u,v \big(f \in u \, \land \,x \in v \to  f(x) \in W[u,v]\big),
\end{eqnarray*}
hence it is enough to construct a closed term $W$ such that
\[
       W[u,v] =_{\tau^*} \langle u_i (v_j) \, : \, i < |u|, j < |v|\rangle,
\]
and this we can do easily using closed terms from $\Tstar$.
\end{enumerate}
\item The external induction axiom $\IA^{\st}$: We will consider the equivalent external induction rule
\[
         \IR^{\st}:\quad \begin{array}{c}\Phi (0),\ \forallst n^0 (\Phi (n) \to \Phi(n+1)) \\  \hline  \forallst n^0 \Phi(n) \end{array},
\]
from which one can derive $\IA^{\st}$ by taking
\[
   \Phi(m^0):\equiv \Psi (0)\, \land \,  \forallst n^0 (\Psi (n) \to \Psi(n+1)) \to \Psi(m).
\]
We assume that we have terms $\tup{T}_1$ and $\tup{T}_2,\tup{T}_3 $ realizing the interpretations of $\Phi(0)$ and $\forallst n^0 (\Phi (n) \to \Phi(n+1))$, and we wish to obtain terms $\tup{T}_4$ realizing the interpretation of $\forallst n^0 \Phi(n) $. So with $\big(\Phi(n)\big)^{\D}\equiv \existsst \tup{x}\forallst \tup{y} \, \phi (\tup{x},\tup{y},n,\tup{a})$ we have
\begin{equation*}
      \ehastar + \Delta_{\intern}  \vdash \forall \tup{y} \, \phi (\tup{T}_1,\tup{y},0,\tup{a})
\end{equation*}
and
\begin{eqnarray*}
        &   \ehastar + \Delta_{\intern} \vdash \forall n^0, \tup{x} ,\tup{y}' \\ &
            \Big(  \forall \tup{i} \in \tup{T}_3 [n,\tup{x}, \tup{y}']  \, \phi(\tup{x},\tup{i}, n,\tup{a} )  \to \phi(\tup{T}_2[n,\tup{x}] , \tup{y}',n+1,\tup{a})\Big),
\end{eqnarray*}
and we want
\begin{equation}\label{eq:proof_induction}
      \ehastar + \Delta_{\intern}  \vdash \forall n^0,\tup{y} \, \phi (\tup{T}_4[n],\tup{y},n,\tup{a}).
\end{equation}
By taking $\tup{T}_4 := \Lambda  \, n^0\, . \,  \tup{\mathcal{R}}(n,\tup{T}_1,\tup{T}_4)$ we get
\begin{eqnarray*}
       \ehastar + \Delta_{\intern} & \vdash & \tup{T}_4[0] =_{\tup{\rho}} \tup{T}_1 \\
       \ehastar + \Delta_{\intern} & \vdash & \tup{T}_4[n+1] =_{\tup{\rho}} \tup{T}_2 [n,\tup{T}_4[n]],
\end{eqnarray*}
which suffices to establish~(\ref{eq:proof_induction}).

\item The principles $\I$, $\NCR$, $\HAC$, $\HIP_{\forallst}$, and $\HGMP$:
\begin{enumerate}
\item $\I$: The $\D$-interpretations of the premise and the conclusion of any instance of $\I$ are identical, and it is easy to show that $\Lambda \tup{x}\, .\, \tup{\langle \tup{x}\rangle}$ realizes the interpretation of the whole implication, provably in $\ehastar$.
\item $\NCR$: Suppose $\big(\Phi(x,y)\big)^{\D}\equiv \existsst \tup{u} \forallst \tup{v} \, \phi (\tup{u},\tup{v},x,y)$. Then
\[
         \big(   \forall x \existsst y \, \Phi(x,y)\big)^{\D}\equiv \existsst \tup{u},y \, \forallst \tup{v} \, \forall x \, \exists i \in y \, \forall \tup{j} \in \tup{v}\, \phi \big(\tup{u},\tup{j},x,i\big)
\]
and
\begin{eqnarray*}
   & \big(   \existsst y \forall x \exists k \in y\, \Phi(x,k)\big)^{\D}  \equiv \\ &   \existsst \tup{u},y \forallst \tup{v} \exists m \in y \forall \tup{n} \in \tup{v}
      \forall x \exists k \in m \forall \tup{l} \in \tup{n}\, \phi \big(\tup{u},\tup{l},x,k\big),
\end{eqnarray*}
and so $ \big(  \forall x \existsst y \, \Phi(x,y)\to  \existsst y \forall x \exists k \in y\, \Phi(x,k)\big)^{\D}$
is
\begin{eqnarray*}
          & \existsst \tup{U}, Y, \tup{V} \, \forallst  \tup{u} ,y,\tup{v} \, \\ &
            \Big(  \forall \tup{i} \in \tup{V} [\tup{u},y, \tup{v}]  \, \forall x \, \exists j \in y \forall \tup{k} \in \tup{i}             \, \phi \big( \tup{u}  ,   \tup{k}  , x , j \big) \to \\ & \exists m \in Y[\tup{u},y] \, \forall \tup{n} \in \tup{v} \,
           \forall \tilde{x} \exists \tilde{k} \in m \, \forall \tup{l} \in \tup{n} \,
 \phi \big( \tup{U} [\tup{u},y ] ,\tup{l} ,\tilde{x}, \tilde{k}\big)
           \Big).
\end{eqnarray*}
Thus we can take
\begin{eqnarray*}
     \tup{U}   &:=& \Lambda \,  \tup{u},y \, . \,  \tup{u},  \\
     Y   &:=& \Lambda \, \tup{u},y \, . \,  \langle y \rangle,  \\
     \tup{V}   &:=& \Lambda \,  \tup{u},y,\tup{v}  \, . \,   \tup{v}.
\end{eqnarray*}
\item $\HAC$: Let $\big(\Phi(x,y)\big)^{\D}\equiv \existsst \tup{u} \forallst \tup{v} \, \phi (\tup{u},\tup{v},x,y)$.
Then $\big(\forallst x \existsst y \,  \Phi (x,y)\big)^{\D}$ is
\[
   \existsst \tup{U},Y  \forallst x,\tup{v} \exists i \in Y[x] \forall \tup{j} \in \tup{v}\, \phi (\tup{U}[x],\tup{j},x,i)
\]
and $\big(\existsst F \forallst x \exists i \in F(x) \, \Phi (x,i)\big)^{\D}$ is
\begin{eqnarray*}
    &    \existsst \tup{\tilde{U}},F \forallst \tilde{x},\tup{\tilde{v}} \exists k \in F \forall l \in \tilde{x}
    \forall \tup{m} \in \tup{\tilde{v}}   \exists i' \in k(l)
 \forall \tup{j'} \in \tup{m}\, \phi \big(\tup{\tilde{U}} [l]   , \tup{j'},  l ,   i'\big),
\end{eqnarray*}
so $\big(\forallst x \existsst y \,  \Phi (x,y)   \to  \existsst F \forallst x \exists i \in F(x) \Phi (x,i)   \big)^{\D}$ is

\begin{eqnarray*}
          & \existsst \tup{\tilde{U}}, F,X, \tup{V} \, \forallst  \tup{U} ,Y,\tilde{x},\tup{\tilde{v}} \\ &
            \Big(
            \forall n \in X[\tup{U},Y,\tilde{x}, \tup{\tilde{v}}] \forall \tup{n}' \in \tup{V} [\tup{U},Y,\tilde{x}, \tup{\tilde{v}}]
           \exists i \in Y[n]
            \forall \tup{j} \in \tup{n}'\, \phi (\tup{U}[n],\tup{j},     n     ,  i) \to \\ &
                     \exists k \in F [\tup{U}, Y] \, \forall l \in \tilde{x} \,
    \forall \tup{m} \in \tup{\tilde{v}}  \, \exists i' \in k(l) \,
   \forall \tup{j'} \in \tup{m}\, \phi \big(\tup{\tilde{U}} [\tup{U}, Y] [l]   , \tup{j'},  l ,   i'\big)
           \Big).
\end{eqnarray*}
Hence we can take
\begin{eqnarray*}
     \tup{\tilde{U}}   &:=& \Lambda \,  \tup{U},Y \, . \,  \tup{U},  \\
      F   &:=& \Lambda \, \tup{U},Y \, . \,  \langle \lambda x. Y[x] \rangle,  \\
     X   &:=& \Lambda \,  \tup{U},Y,\tilde{x},\tup{\tilde{v}} \, . \,  \tilde{x},  \\
     \tup{V}   &:=& \Lambda \,  \tup{U},Y,\tilde{x},\tup{\tilde{v}}  \, . \,   \tup{\tilde{v}}.
\end{eqnarray*}

\item $\HIP_{\forallst}$: Let $\big(\Psi(y)\big)^{\D}\equiv \existsst \tup{u} \forallst \tup{v} \, \psi (\tup{u},\tup{v},y)$. Then $ \big( \forallst x \, \phi(x) \to\existsst y\Psi(y)\big)^{\D}$ is
\[
       \existsst \tup{U},Y,X \forallst \tup{v}\, \Big(\forall k \in X [\tup{v}] \, \phi(k) \to
       \exists i \in Y \forall \tup{j} \in \tup{v} \, \psi \big(\tup{U}, \tup{j} ,i\big)\Big)
\]
and $\big(\existsst y \, \big(\forallst x \,  \phi (x)\to \exists i \in y \Psi(i)\big)\big)^{\D}$ is
\begin{eqnarray*}
    &    \existsst \tup{U},X,y \forallst \tup{v} \exists n \in y \forall \tup{m} \in \tup{v} \,
     \Big( \forall k \in X[\tup{m}] \, \phi(k) \to \exists i \in n \, \forall \tup{j} \in \tup{m}\, \psi \big(\tup{U},\tup{j},i\big)\Big),
\end{eqnarray*}
so the $\D$-interpretation of $\HIP_{\forallst}$ is
%\begin{eqnarray*}
%       &&
%        \existsst \tup{\tilde{U}},\tilde{X},\tilde{Y},\tup{V} \forallst \tup{U},Y,X,\tup{\tilde{v}}\, \Big(
%        \forall \tup{l}\leq |\tup{V} [\tup{U}, Y, X, \tup{\tilde{v}}]|
%        \\ &&
%        \quad \big( \forall k \leq |X[(    \tup{V} [\tup{U}, Y, X, \tup{\tilde{v}}]   )_{\tup{l}}] |
%         \phi \big((   X[(    \tup{V} [\tup{U}, Y, X, \tup{\tilde{v}}]   )_{\tup{l}}]      )_k\big)
%          \\ && \qquad  \to
%       \exists i \leq |Y| \forall \tup{j} \leq |   (    \tup{V} [\tup{U}, Y, X, \tup{\tilde{v}}]   )_{\tup{l}}     |
%       \,
%        \psi \big(\tup{U}, (   (    \tup{V} [\tup{U}, Y, X, \tup{\tilde{v}}]   )_{\tup{l}}      )_{\tup{j}} ,(Y)_i\big)\big)
%       \\
%    && \quad \to
%    \\
%    &&
%    \qquad  \exists \tilde{n}\leq |\tilde{Y} [\tup{U},Y,X] | \forall \tup{\tilde{m}} \leq |\tup{\tilde{v}}|
%     \big(
%     \\ &&  \qquad \quad  \forall \tilde{k} \leq |\tilde{X}[ \tup{U},Y,X ]
%[(\tup{\tilde{v}})_{\tup{\tilde{m}}}]| \, \phi\big(( \tilde{X}[ \tup{U},Y,X ]
% [(\tup{\tilde{v}})_{\tup{\tilde{m}}}]  )_{\tilde{k}} \big)
%      \\ &&
%     \qquad \quad \to \exists \tilde{i} \leq |(  \tilde{Y} [\tup{U},Y,X]   )_{\tilde{n}}|\,
%\forall \tup{\tilde{j}} \leq |(\tup{\tilde{v}})_{\tup{\tilde{m}}}|
%     \\ &&
%      \qquad \qquad  \psi \big(\tup{\tilde{U}}   [\tup{U},Y,X]   ,
%  ((\tup{\tilde{v}})_{\tup{\tilde{m}}})_{\tup{\tilde{j}}},
%  (  (  \tilde{Y} [\tup{U},Y,X]   )_{\tilde{n}} )_{\tilde{i}} \big)\big)\Big).
%\end{eqnarray*}
\begin{eqnarray*}
       &
        \existsst \tup{\tilde{U}},\tilde{X},\tilde{Y},\tup{V} \forallst \tup{U},Y,X,\tup{\tilde{v}}\, \\ &
 \Big(
        \forall \tup{l} \in \tup{V} [\tup{U},Y,X, \tup{\tilde{v}}] \, \big( \forall k \in X[\tup{l}] \,
         \phi (k) \to
       \exists i \in Y \, \forall \tup{j} \in \tup{l}       \,
        \psi \big(\tup{U}, \tup{j} ,i\big)\big)
   \to \\ &  \exists \tilde{n} \in \tilde{Y} [\tup{U},Y,X]  \forall \tup{\tilde{m}} \in \tup{\tilde{v}} \,
     \big( \,
     \forall \tilde{k} \in \tilde{X}[\tup{U},Y,X ][\tup{\tilde{m}}] \varphi(\tilde{k}) \to \\ & \exists \tilde{i}\in \tilde{n} \, \forall \tup{\tilde{j}} \in \tup{\tilde{m}} \, \psi \big(\tup{\tilde{U}}[\tup{U},Y,X], \tup{\tilde{j}}, \tilde{i} \big)\big)\Big).
\end{eqnarray*}
Hence we can take
\begin{eqnarray*}
     \tup{\tilde{U}}   &:=& \Lambda \,  \tup{U},Y,X \, . \,  \tup{U},  \\
     \tilde{X}   &:=& \Lambda \,  \tup{U},Y,X \, . \,  X,  \\
     \tilde{Y}   &:=& \Lambda \, \tup{U},Y,X \, . \,  \langle Y \rangle,  \\
     \tup{V}   &:=& \Lambda \,  \tup{U},Y,X,\tup{\tilde{v}}  \, . \,   \tup{\tilde{v}}.
\end{eqnarray*}

\item $\HGMP$: The $\D$-interpretation of $\HGMP$ is
\begin{eqnarray*}
    &  \existsst X \forallst x \, \Big(\big(\forall x' \in x \, \phi(x') \to \psi \big) \to \exists j\in X[x] \big(\forall k \in j \, \phi(k) \to \psi \big)\Big),
\end{eqnarray*}
and we can take $X := \Lambda  \, x\, . \,  \langle x\rangle$.

\end{enumerate}

\end{enumerate}

\end{proof}

Theorem~\ref{soundnessDst} implies the following conservation result, which improves on Moerdijk and Palmgren~\cite{moerdijkpalmgren97} and Avigad and Helzner~\cite{avigadhelzner02}.

\begin{cor} The system
\[
    \ehaststar + \I + \NCR + \HAC + \HGMP
    +  \HIP_{\forallst}
\]
is a conservative extension of $\ehastar$ and hence of $\eha$.

\end{cor}

\begin{proof}
By Theorem~\ref{soundnessDst} and Lemma~\ref{le:forallst-formulas}.

\end{proof}

\begin{remark} \label{hanst} We could define a system $\ehanststar$ by adding primitive predicates $\nst^{\sigma}$ (``nonstandard'') to $\ehaststar$ for each finite type $\sigma$, along with axioms
\[
          \forall x^{\sigma} \big(\nst (x) \leftrightarrow \lnot \st (x)\big).
\]
If we then extend the $\D$-interpretation by
\[
      \big(\nst^{\sigma}(x^{\sigma})\big)^{\D} :\equiv \forallst y^{\sigma}  y \neq_{\sigma} x),
\]
we get an analogue of Theorem~\ref{soundnessDst}, since $\big(  \nst (x) \to \lnot \st (x) \big)^{\D}$
is provably equivalent to
\[
     \existsst Y \forallst z \left( \forall y \in Y[z] (y \neq x) \to  x \not\in z \right)
\]
and $\big( \lnot \st (x) \to \nst(x)\big)^{\D}$
to
\[
     \existsst Z \forallst y \left( \forall z' \in Z[y] x \not\in z' \to y\neq x \right),
\]
so that we can take $Y[z] :=z$ and $Z[y] := \langle \langle y \rangle \rangle$  respectively .

\end{remark}

\subsection{The characteristic principles of the nonstandard functional interpretation}

In this section we will prove that the characteristic principles of the nonstandard functional interpretation are $\I$, $\NCR$, $\HAC$, $\HIP_{\forallst}$, and $\HGMP$. For notational simplicity we will let
\[
     \HH := \ehaststar + \I + \NCR + \HAC + \HIP_{\forallst} + \HGMP.
\]
\begin{thm}[Characterization theorem for the nonstandard functional interpretation] $\, $

\begin{enumerate}

\item For any formula $\Phi$ in the language of $\ehaststar$ we have
%\[
%          \ehaststar + \I + \NCR + \HAC + \HIP_{\forallst} + \HGMP \vdash \Phi \leftrightarrow \Phi^{\D}.
%\]
\[
          \HH \vdash \Phi \leftrightarrow \Phi^{\D}.
\]
\item For any formula $\Psi$ in the language of $\ehaststar$ we have: If for all $\Phi$ in
$\mathcal{L}(\ehaststar)$ (with $\Phi^{\D}\equiv \existsst \tup{x}\forallst \tup{y}\, \phi_{\D}(\tup{x},\tup{y})$) the implication
\begin{equation}\label{eq:char_thm_dst}
          \HH +\Psi \vdash \Phi  \quad \Longrightarrow \quad \mbox{there are closed terms $\tup{t}\in\Tstar$ s.t. }
          \eha \vdash \forall \tup{y}\, \phi_{\D}(\tup{t},\tup{y})
\end{equation}
holds, then $\HH \vdash \Psi$.
%implies that there exist closed terms $\tup{T}$ in G\"{o}del's $\T$ such that
%\[
%         \eha \vdash \forall \tup{y}\, \phi(\tup{x},\tup{y}),
%\]
%where $\Phi^{\D}\equiv \existsst \tup{x}\forallst \tup{y}\, \phi(\tup{x},\tup{y})$, we have
%\[
%          \HH \vdash \Psi
%\]

\end{enumerate}
\end{thm}
\begin{proof}
\begin{enumerate}
\item
We will prove item 1 by induction on the logical structure of $\Phi$.
      \begin{enumerate}
          \item As induction start we note that for internal atomic $\phi$ we obviously have $\HH \vdash \phi \leftrightarrow \phi^{\D}$,
%since $\phi \equiv \phi^{\D}$,
that
\[
         \HH \, \vdash \, \st(u^{\sigma}) \to \existsst x^{\sigma^*} u \in x
\]
follows by taking $x:=\langle u \rangle$, and that
\[
         \HH \, \vdash \, \existsst x^{\sigma^*} u \in x \to  \st(u^{\sigma})
\]
is Lemma~\ref{elemstsetset}.
          \item For the induction step involving $\Phi$, $\Psi$ we will use that (see 1.6.17 in~\cite{Troelstra73}) via appropriate embeddings of tuples of types in a suitable common higher type and tuple coding of functionals with inverses (all given by terms in $\Tstar$) there are $\phi(x,y)$, $\psi(u,v)$ such that
\[
     \ehaststar  \, \vdash \, \Phi^{\D} \leftrightarrow \existsst x \forallst y \, \phi(x,y)
\]
and
\[
     \ehaststar  \, \vdash \, \Psi^{\D} \leftrightarrow \existsst u \forallst v \, \psi(u,v),
\]
and such that $\phi(x,y)$ and $\psi(u,v)$ are provably upwards closed in $x$ and $u$ respectively.
            \begin{enumerate}
                      \item For $\land$ we must consider
                      \[\existsst x \forallst y \, \phi(x,y) \, \land \, \existsst u \forallst v \, \psi(u,v) \, \leftrightarrow \, \existsst x,u \forallst y,v \, \big(\phi(x,y) \, \land \, \psi(u,v)\big),\]
                      which follows by intuitionistic logic. (We can assume that $u,v$ do not appear in $\phi$, and that $x,y$ do not appear in $\psi$.)
                      \item For $\lor$ we note that
                       \[\existsst x \forallst y \, \phi(x,y) \, \lor \, \existsst u \forallst v \, \psi(u,v) \, \leftrightarrow \, \existsst x,u \forallst y,v \, \big(\phi(x,y) \, \lor \, \psi(u,v)\big)\]
                       follows from $\LLPO$, which by Proposition~\ref{prop:LLPO} follows from $\I$.
                      \item For implication we use that $\psi (u,v)$ is upwards closed in $u$ to conclude
                      \begin{eqnarray*}
                    \existsst x \forallst y \, \phi (x,y) \to \existsst u \forallst v \,\psi (u,v) & \leftrightarrow & \\
              \forallst x \, \big(\forallst y \, \phi (x,y) \to \existsst u \forallst v \,\psi (u,v)\big) &   \stackrel{\HIP_{\forallst}}{\leftrightarrow}& \\
               \forallst x \existsst u\, \big(\forallst y \, \phi (x,y) \to \forallst v \,\psi (u,v)\big) &  \leftrightarrow & \\
               \forallst x \existsst u \forallst v\, \big(\forallst y \, \phi (x,y) \to \psi (u,v)\big) &  \stackrel{\HGMP}{\leftrightarrow} & \\
                \forallst x \existsst u \forallst v \existsst y\, \big(\forall i\in y \, \phi (x,i) \to \psi (u,v)\big) &  \stackrel{\HAC}{\leftrightarrow} & \\
                 \existsst U \forallst x \forallst v \existsst y\, \big(\forall i\in y \, \phi (x,i) \to \psi (U(x),v)\big) &  \stackrel{\HAC +\tiny{\mbox{ coding }} x,v \tiny{\mbox{ into one}}}{\leftrightarrow} & \\
                  \existsst U,Y \forallst x , v \, \big(\forall i\in Y(x,v) \, \phi (x,i) \to \psi (U(x),v)\big) & \leftrightarrow & \\
                    \existsst U,Y \forallst x , v \, \big(\forall i\in Y[x,v] \, \phi (x,i) \to \psi (U[x],v)\big). &&
                  \end{eqnarray*}

                      \item For $\forall$ we use $\NCR$ and that $\phi(x,y,z)$ is upwards closed in $x$ to get
                      \[
                       \forall z \existsst x \forallst y \, \phi (x,y,z)  \ \stackrel{\NCR}{\leftrightarrow} \
     \existsst x \forall z \forallst y \, \phi (x,y,z)    \  \leftrightarrow \
       \existsst x  \forallst y \forall z \, \phi (x,y,z).
                      \]

                      \item For $\exists$ we use $\I$:
                       \begin{eqnarray*}
                   \exists z \existsst x \forallst y \, \phi (x,y,z)  & \leftrightarrow& \\
     \existsst x \exists z \forallst y \, \phi (x,y,z)    &  \stackrel{\I}{\leftrightarrow}& \\
       \existsst x  \forallst y \exists z \forall y' \in y\, \phi (x,y',z).    &  &
                  \end{eqnarray*}

                      \item For $\forallst$ we use $\HAC$ and that $\phi(x,y,z)$ is upwards closed in $x$:
                      \[
                   \forallst z \existsst x \forallst y \, \phi (x,y,z)  \   \stackrel{\HAC}{\leftrightarrow} \
       \existsst X  \forallst z,y  \, \phi (X(z),y,z)     \  \leftrightarrow \
       \existsst X  \forallst z,y  \, \phi (X[z],y,z).
                  \]

                      \item For $\existsst$ we will again use $\I$:
                      \begin{eqnarray*}
                   \existsst z \big(\existsst x \forallst y \, \phi (x,y,z)\big)  & \leftrightarrow & \\
     \existsst z \exists z' \in z \big(\existsst x\forallst y \, \phi (x,y,z')\big)    &  \leftrightarrow & \\
     \existsst z \, \existsst x \, \exists z' \in z \, \forallst y \, \phi (x,y,z')    & \stackrel{\I}{\leftrightarrow}& \\
     \existsst z \, \existsst x \, \forallst y \, \exists z' \in z \, \forall y' \in y \, \phi (x,y',z').    & &
                  \end{eqnarray*}

            \end{enumerate}
    \end{enumerate}
\item Assume that $\Psi$ is a formula of $\mathcal{L}(\ehaststar)$
such that implication~(\ref{eq:char_thm_dst})
%the implication in item 2
holds, and let
$\Psi^{\D}\equiv \existsst \tup{x}\forallst \tup{y}\, \psi_{\D}(\tup{x},\tup{y})$. Then because of item 1 we have
\begin{eqnarray*}
            &   \HH +\Psi   \vdash  \Psi  & \\
                                  & \Longrightarrow  &  \\
    &   \exists \tup{t}\in\Tstar \mbox{ s.t. }    \ehastar  \vdash   \forall \tup{y}\, \psi_{\D}(\tup{t},\tup{y}) & \\
                                  & \Longrightarrow   & \\
      &    \exists \tup{t}\in\Tstar \mbox{ s.t. }     \ehaststar  \vdash   \forall \tup{y}\, \psi_{\D}(\tup{t},\tup{y}) & \\
                                  & \Longrightarrow  & \\
       &   \exists \tup{t}\in\Tstar \mbox{ s.t. }     \ehaststar  \vdash   \forallst \tup{y}\, \psi_{\D}(\tup{t},\tup{y}) & \\
                                  & \Longrightarrow & \\
             &   \ehaststar  \vdash   \existsst \tup{x}\forallst \tup{y}\, \psi_{\D}(\tup{x},\tup{y}) & \\
                                  & \Longrightarrow  & \\
            &   \HH  \vdash   \existsst \tup{x}\forallst \tup{y}\, \psi_{\D}(\tup{x},\tup{y}) & \\
                                  & \Longrightarrow  & \\
            &   \HH  \vdash  \Psi. &
\end{eqnarray*}
%\begin{eqnarray*}
%            &   \HH +\Psi   \vdash  \Phi
%                                  & \Longrightarrow  \\
%          \exists \tup{t}\in\T \mbox{ s.t. }  &   \eha  \vdash   \forall \tup{y}\, \phi_{\D}(\tup{t},\tup{y})
%                                  & \Longrightarrow  \\
%          \exists \tup{t}\in\T \mbox{ s.t. }  &   \ehaststar  \vdash   \forall \tup{y}\, \phi_{\D}(\tup{t},\tup{y})
%                                  & \Longrightarrow  \\
%          \exists \tup{t}\in\T \mbox{ s.t. }  &   \ehaststar  \vdash   \forallst \tup{y}\, \phi_{\D}(\tup{t},\tup{y})
%                                  & \Longrightarrow  \\
%             &   \ehaststar  \vdash   \existsst \tup{x}\forallst \tup{y}\, \phi_{\D}(\tup{x},\tup{y})
%                                  & \Longrightarrow  \\
%            &   \HH  \vdash   \existsst \tup{x}\forallst \tup{y}\, \phi_{\D}(\tup{x},\tup{y})
%                                  & \Longrightarrow  \\
%            &   \HH  \vdash  \Phi &
%\end{eqnarray*}
\end{enumerate}
\end{proof}

Theorem~\ref{soundnessDst} allows us to extract a finite sequence of candidates for the existential quantifier in formulas of the form $\forallst x \, \existsst y \, \phi (x,y)$, in the following sense:

\begin{thm}[Main theorem on program extraction by the $\D$-interpretation]
Let $\forallst x \existsst y \, \phi (x,y)$ be a sentence of $\ehaststar$ with $\phi(x,y)$ an internal formula,
and let $\Delta_{\intern}$ be a set of internal sentences.
If
\[
    \ehaststar + \I + \NCR + \HAC + \HGMP
    +  \HIP_{\forallst} + \Delta_{\intern} \vdash \forallst x \, \existsst y \, \phi (x,y),
\]
then from the proof we can extract a closed term $t$ in $\Tstar$ such that
\[
    \ehastar +\Delta_{\intern} \vdash \forall x \, \exists y \in t(x)\, \phi (x,y).
\]

\end{thm}
\begin{proof}
Since
\[
\big(\forallst x \existsst y \, \phi (x,y)\big)^{\D} \equiv \existsst Y \forallst x \exists y \in Y[x] \, \phi(x,y)
\]
it follows from the soundness theorem of the $\D$-interpretation that there is a closed term $s$ such that
\[
    \ehastar +\Delta_{\intern} \vdash \forall x \exists y \in s[x]\, \phi (x,y),
\]
and so we can let $t:=\lambda x \, . \, s[x]$.
\end{proof}

\begin{remark}
Probably it is clear by now that our functional interpretation has some striking similarities with the bounded functional interpretation due to Ferreira and Oliva \cite{ferreiraoliva05}. Also there the authors work with two types of quantifiers, get an interpretation whose matrix is upwards closed in the first component (albeit with respect to a different ordering), interpret implications \emph{\`a la} Diller-Nahm and have some similar looking characteristic principles, like a monotone axiom of choice. But, still, the precise relationship is not entirely clear to us, because we are now comparing the external quantifiers with the unbounded quantifiers in the bounded functional interpretation, which is not very natural. These issues deserve to be further investigated.
\end{remark}

\subsection{Discussion}

It follows from the soundness of the $\D$-interpretation (Theorem \ref{soundnessDst}) that it can be used to eliminate nonstandard principles, like overspill, realization and idealization, from proofs. It also allows one to eliminate underspill, since we have the following result (recall that $\R$ is the realization principle from Section 4.1):
\begin{prop} \label{USfromRHGMP} We have
\[ \ehaststar + \R + \HGMP \vdash \US, \]
and therefore the underspill principle $\US$ is eliminated by the $\D$-interpretation.
\end{prop}
\begin{proof}
We reason in $\ehaststar + \R + \HGMP$. Assume $\forall x^\sigma ( \lnot \st(x) \to \varphi(x))$. Our aim is to find a standard $x$ such that $\varphi(x)$ holds.

$\forall x^\sigma \big( \lnot \st(x) \to \varphi(x) \big)$ is equivalent to $\forall x^\sigma \big( \, \forallst y^\sigma (y \not= x) \to \varphi(x) \, \big)$, which, using $\HGMP$, we can rewrite as
\[ \forall x^\sigma \, \existsst y^{\sigma^*} \big( \forall y' \in y (y' \not= x) \to \varphi(x) \big). \]
It now follows from $\R$ that there is a standard $y$ of type $\sigma^*$ such that
\[ \forall x^\sigma \big( \forall y' \in y (y' \not= x) \to \varphi(x) \big). \]
So if we choose $x$ to be a standard element of type $\sigma$ different from all $y' \in y$, then $\varphi(x)$ will hold.
\end{proof}

We also have:
\begin{prop}
The system $\HH :\equiv \ehaststar + \I + \NCR + \HAC + \HGMP +  \HIP_{\forallst}$ is closed under both transfer rules, $\TRA$ and $\TRE$.
\end{prop}
\begin{proof}
If $\HH \vdash \forallst x \, \varphi(x)$, then $\ehastar \vdash \forall x \, \varphi(x)$ by soundness of the $\D$-interpretation. Since $\ehastar$ is a subsystem of $\HH$, it follows that $\HH \vdash \forall x \, \varphi(x)$. This shows that $\HH$ is closed under $\TRA$.

If $\HH \vdash \exists x \, \varphi(x)$, then $\ehastar \vdash \exists x \, \varphi(x)$ by conservativity of $\HH$ over $\ehastar$. Since $\ehastar$ has the existence property (the usual argument for the existence property of $\eha$, as in \cite[Corollary 5.24]{Kohlenbach08}, for instance, carries over to $\ehastar$), it follows that there is a term $t$ in $\Tstar$ such that $\ehastar \vdash \varphi(t)$. Again, because $\ehastar$ is a subsystem of $\HH$, we have $\HH \vdash \varphi(t)$ as well, and since all terms from $\Tstar$ are provably standard in $\HH$, we have $\HH \vdash \existsst x \, \varphi(x)$. This shows that $\HH$ is closed under $\TRE$ as well.
\end{proof}

Therefore our functional interpretation $\D$ meets all the benchmarks that we discussed in Section 4.

% The connectives whose interpretation look different are disjunction (but for that see Section 9 below) and the external universal quantifier, which, in the bounded functional interpretation, yields a functional $X$ which works on a bound on $z$ rather than on $z$ itself. In our context, this makes no difference, because if $z$ codes some sequence $\langle z_1, \ldots, z_n \rangle$, then we can replace $X$ with $X(z_1) * \ldots * X(z_n)$ and use the upwards closure of $\varphi$ in $x$. 
\section{The system $\epaststar$ and negative translation}

By combining the functional interpretation from the previous section with negative translation we can obtain conservation and term extraction results for classical systems as well. We will work out the details in this and the next section.

First, we need to set up a suitable classical system $\epaststar$. It will be an extension of $\epastar$, which is $\ehastar$ with the law of excluded middle added for all formulas. When working with classical systems, we will often take the logical connectives $\lnot, \lor, \forall$ as primitive and regard the others as defined. In a similar spirit, the language of $\epaststar$ will be that of $\epastar$ extended just with unary predicates $\st^\sigma$ for every type $\sigma \in \Tpstar$; the external quantifiers $\forallst, \existsst$ are regarded as abbreviations:
\begin{eqnarray*}
\forallst x \, \Phi(x) & :\equiv &  \forall x ( \, \st(x)\rightarrow\Phi(x) \, ),\\
\existsst x \, \Phi(x) & :\equiv & \exists x (\, \st(x)\wedge\Phi(x) \, ).
\end{eqnarray*}

\begin{dfn}[$\epaststar$] The system $\epaststar$ is
\[\epaststar := \epastar + \Tst + \IA^{\st{}} \]
where
\begin{itemize}
\item $\Tst$ consists of:
\begin{enumerate}
\item the schema $\st(x) \land x = y \to \st(y)$,
\item a schema providing for each closed term $t$ in $\Tstar$ the axiom $\st(t)$,
\item the schema $\st(f)\wedge\st(x)\rightarrow\st(fx)$.
\end{enumerate}
\item $ \IA^{\st{}}$ is the external induction axiom:
\[
\IA^{\st{}} \quad : \quad\big(\Phi(0)\wedge\forallst n^0 (\Phi(n)\rightarrow\Phi(n+1) )\big)\rightarrow\forallst n^0 \Phi(n).
\]
\end{itemize}
Again we warn the reader that the induction axiom from $\epastar$
\[ \quad\big(\varphi(0)\wedge\forall n^0 (\varphi(n)\rightarrow\varphi(n+1) )\big)\rightarrow\forall n^0 \varphi(n) \]
is supposed to apply to internal formulas $\varphi$ only.
\end{dfn}

As for $\ehaststar$, we have:
\begin{prop}
If a formula $\Phi$ is provable in $\epaststar$, then its internalization $\Phi^{\intern}$ is provable in $\epastar$. Hence $\epaststar$ is a conservative extension of $\epastar$ and $\epa$.
\end{prop}

We will now show how negative translation provides an interpretation of $\epaststar$ in $\ehaststar$. Various negative translations exist, with the one due to G\"odel and Gentzen being the most well-known. Here, we work with two variants, the first of which is due to Kuroda \cite{kuroda51}.

\begin{dfn}[Kuroda's negative translation for $\epaststar$]
For an arbitrary formula $\Phi$ in the language of $\epast$, we define
its Kuroda negative translation in $\ehaststar$ as
\[
\Phi^\ku\ :\equiv\ \neg\neg\Phi_\ku,
\]
where $\Phi_\ku$ is defined inductively on the structure of $\Phi$ as follows:
\begin{align*}
 \Phi_\ku   &:\equiv \Phi\quad \text { for atomic formulas $\Phi$}, \\
%t&\hr \st(s) & &:\equiv& &t=\langle t_0, \ldots, t_n\rangle\text{ and $s=t_i$ for some $i\leq n$},\\
 \big(\neg\Phi\big)_\ku  &:\equiv \neg\Phi_\ku,\\
 \big(\Phi\vee\Psi\big)_\ku  &:\equiv  \Phi_\ku\vee \Psi_\ku,\\
 \big(\forall x \, \Phi(x)\big)_\ku  &:\equiv  \forall x \, \neg\neg\Phi_\ku(x).
\end{align*}
\end{dfn}

\begin{thm} \label{soundnesskuroda} $\epaststar \vdash \Phi \leftrightarrow \Phi^\ku$ and if $\epaststar + \Delta \vdash\Phi$ then $\ehaststar + \Delta^\ku \vdash\Phi^\ku$.
\end{thm}
\begin{proof}
It is clear that, classically, $\Phi$, $\Phi_\ku$ and $\Phi^\ku$ are all equivalent. The second statement is proved by induction on the proof of $\epaststar + \Delta \vdash\Phi$. For the cases of the axioms and rules of classical logic and $\epastar$, see, for instance, \cite[Proposition 10.3]{Kohlenbach08}. As the Kuroda negative translation of every instance of $\Tst$ or $\IA^{\st{}}$ is provable in $\ehaststar$ using the same instance of $\Tst$ or $\IA^{\st{}}$, the statement is proved.
\end{proof}

It will turn out to be convenient to introduce a second negative translation, extracted from the work of Krivine by Streicher and Reus (see \cite{krivine90, streicherreus98, streicherkohlenbach07}). This translation will interpret $\epaststar$ into $\ehanststar$ (see Remark \ref{hanst}).

%The formal definition of the interpretation is as follows.
\begin{dfn}[Krivine's negative translation for $\epaststar$]
For an arbitrary formula $\Phi$ in the language of $\epaststar$, we define its Krivine negative translation in $\ehanststar$ as
\[
\Phi^\kr\ :\equiv\ \neg\Phi_\kr,
\]
where $\Phi_\kr$ is defined inductively on the structure of $\Phi$ as follows
\begin{align*}
 \phi_\kr   &:\equiv \neg\phi\quad \text { for an internal atomic formula $\phi$}, \\
 \st(x)_\kr  &:\equiv \nst(x), \\
%t&\hr \st(s) & &:\equiv& &t=\langle t_0, \ldots, t_n\rangle\text{ and $s=t_i$ for some $i\leq n$},\\
 \big(\neg\Phi\big)_\kr  &:\equiv \neg\Phi_\kr,\\
 \big(\Phi\vee\Psi\big)_\kr  &:\equiv \Phi_\kr\wedge \Psi_\kr,\\
 \big(\forall x \, \Phi(x)\big)_\kr  &:\equiv  \exists x \, \Phi_\kr(x).\\
\end{align*}
\end{dfn}

\begin{thm}\label{l:kukr}
For every formula $\Phi$ in the language of $\epaststar$, we have:
\begin{enumerate}
\item $\ehanststar\vdash\Phi^\kr\leftrightarrow\Phi^\ku$.
\item If $\epaststar + \Delta \vdash \Phi$, then $\ehanststar + \Delta ^\kr \vdash \Phi^\kr$.
\end{enumerate}
\end{thm}
\begin{proof}
Item 1 is easily proved by induction on the structure of $\Phi$. Item 2 follows from item 1 and Theorem \ref{soundnesskuroda}.
\end{proof}

\section{A functional interpretation for $\epaststar$} \label{s:Shoenfield}

We will now combine negative translation and our functional interpretation $\D$ to obtain a functional interpretation of the classical system $\epaststar$.

\subsection{The interpretation}

\begin{dfn} \label{d:stS} ($\Sh$-interpretation for $\epaststar$.) To each formula $\Phi(\tup{a})$ with free variables $\tup{a}$ in the language of $\epaststar$ we associate its {$\Sh$-interpretation}
\[
\Phi^\Sh(\tup{a}):\equiv\forallst \tup{x}\,\existsst \tup{y}\,
\phi_\Shb(\tup{x},\tup{y}, \tup{a})
\text{,}
\]
where $\phi_\Shb$ is an internal formula. Moreover,  $\tup{x}$ and $\tup{y}$ are tuples of variables whose length and types depend only on the logical structure of $\Phi$. The interpretation of the formula is defined inductively on its
structure. If
\[
\Phi^\Sh(\tup{a}):\equiv\forallst \tup{x}\,\existsst \tup{y}\, \
\phi_\Shb(\tup{x},\tup{y}, \tup{a}) \
\text{ and } \ \Psi^\Sh(\tup{b}):\equiv\forallst \tup{u}\,\existsst \tup{v}\, \
\psi_\Shb(\tup{u},\tup{v}, \tup{b}),
\]
then
\begin{enumerate}
\item[(i)] $\phi^\Sh
:\equiv \phi$ for atomic internal
$\phi(\tup{a}),$
\item[(ii)] $\big(\st(z) \big)^\Sh :\equiv \existsst x \, ( z = x)$,
\item[(iii)] $(\neg \Phi)^\Sh :\equiv \forallst \tup{Y} \existsst \tup{x} \,
\forall \tup y \in \tup Y[\tup x] \neg \phi_\Shb(\tup{x}, \tup y, \tup a),$
\item[(iv)] $(\Phi \vee \Psi)^\Sh :\equiv
             \forallst \tup{x},\tup{u} \existsst \tup{y},\tup{v} \,
\big(\phi_\Shb(\tup{x},\tup{y}, \tup a) \vee \psi_\Shb(\tup{u},\tup{v}, \tup b)\big),$
\item[(v)] $(\forall z \, \phi)^\Sh :\equiv \forallst \tup{x}
\existsst \tup{y} \forall z \exists {\tup y}' \in \tup y \, \phi_\Shb(\tup x,\tup{y}', z).$
\end{enumerate}
\end{dfn}

\begin{thm} \label{soundnessshoenfield} {\rm (Soundness of the $\Sh$-interpretation.)} Let $\Phi(\tup a)$ be a formula in the language of $\epaststar$ and suppose $\Phi(\tup a)^\Sh\equiv\forallst \tup x \, \existsst \tup y \, \phi(\tup x, \tup y, \tup a)$. If $\Delta_{\intern}$ is a collection of internal formulas and
\[ \epaststar + \Delta_{\intern} \vdash \Phi(\tup a), \]
then one can extract from the formal proof a sequence of closed terms $\tup t$ in $\Tstar$ such that
\[
\epastar + \Delta_{\intern} \vdash\ \forall \tup x\exists \tup y \in \tup t(\tup x)\ \phi(\tup x,\tup y, \tup a).
\]
\end{thm}

Our proof of this theorem relies on the following lemma:
\begin{lemma}\label{p:ShD}
Let $\Phi(\tup a)$ be a formula in the language of $\epaststar$ and assume
\begin{eqnarray*}
\Phi^\Sh & \equiv & \forallst \tup x\existsst  \tup y \, \phi(\tup x, \tup y, \tup a) \quad \mbox{and} \\
(\Phi_\kr)^\D & \equiv & \existsst \tup u \forallst \tup v \, \theta(\tup u,\tup v, \tup a).
\end{eqnarray*}
Then the tuples $\tup x$ and $\tup u$ have the same length and the variables they contain have the same types. The same applies to $\tup y$ and $\tup v$. In addition, we have
\[
\epastar\ \vdash\ \phi(\tup x, \tup y, \tup a)\leftrightarrow \lnot \theta(\tup x, \tup y, \tup a).
\]
\end{lemma}
\begin{proof} The proof is by induction on the structure of $\Phi$.
\begin{enumerate}
\item[(i)] If $\Phi \equiv \psi$, an internal and atomic formula, then $\varphi \equiv \psi$ and $\theta \equiv \lnot \psi$, so $\epastar \vdash \varphi \leftrightarrow \lnot \theta$.
\item[(ii)] If $\Phi \equiv \st(z)$, then $\varphi \equiv y = z$ and $\theta \equiv y \not= z$, so $\epastar \vdash \varphi \leftrightarrow \lnot \theta$.
\item[(iii)] If $\Phi \equiv \lnot \Phi'$ with $(\Phi')^\Sh \equiv \forallst \tup x\existsst  \tup y \, \phi'(\tup x, \tup y, \tup a)$ and $(\Phi'_\kr)^\D \equiv \existsst \tup u \forallst \tup v \, \theta'(\tup u,\tup v, \tup a)$, then $\varphi \equiv \forall \tup y' \in \tup Y[\tup x]\ \neg\phi'(\tup x,\tup y')$ and $\theta \equiv \neg\forall \tup i\in \tup Y[\tup x]\ \theta'(\tup x,\tup i)$. Since $\epastar \vdash \varphi' \leftrightarrow \lnot \theta'$ by induction hypothesis, also $\epastar \vdash \varphi \leftrightarrow \lnot \theta$.
\item[(iv)] If $\Phi \equiv \Phi_0 \lor \Phi_1$ with \[ \Phi_i^\Sh \equiv \forallst \tup x\existsst  \tup y \, \phi_i(\tup x, \tup y, \tup a) \] and \[ ((\Phi_i)_\kr)^\D \equiv \existsst \tup u \forallst \tup v \, \theta_i(\tup u,\tup v, \tup a),\] then $\varphi \equiv \varphi_0 \lor \varphi_1$ and $\theta \equiv \theta_0 \land \theta_1$. Since $\epastar \vdash \varphi_i \leftrightarrow \lnot \theta_i$ by induction hypothesis, also $\epastar \vdash \varphi \leftrightarrow \lnot \theta$.
\item[(v)] If $\Phi \equiv \forall z \, \Phi'$ with \[ (\Phi')^\Sh \equiv \forallst \tup x\existsst  \tup y \, \phi'(\tup x, \tup y, z, \tup a) \] and \[ (\Phi'_\kr)^\D \equiv \existsst \tup u \forallst \tup v \, \theta'(\tup u,\tup v, z, \tup a),\] then $\varphi \equiv \forall z \exists \tup y' \in \tup y \varphi'(\tup x, \tup y', z, \tup a)$ and $\theta \equiv \exists z\forall  \tup y' \in \tup y\ \theta'(\tup x,\tup y',z, \tup a)$. Since $\epastar \vdash \varphi' \leftrightarrow \lnot \theta'$ by induction hypothesis, also $\epastar \vdash \varphi \leftrightarrow \lnot \theta$.
\end{enumerate}
\end{proof}

\begin{remark}
This lemma is the reason why we introduced the system $\ehanststar$ in Remark \ref{hanst}: it would fail if we would let the Krivine negative translation land directly in $\ehaststar$ with $\st(z)_\kr = \lnot \st(z)$. As it is, this lemma yields a quick proof of the soundness of the $\Sh$-interpretation.
\end{remark}

\begin{proof} (Of the soundness of the $\Sh$-interpretation, Theorem \ref{soundnessshoenfield}.) Let $\Phi(\tup a)$ be a formula in the language of $\epaststar$ and let  $\varphi$ and $\theta$ be such that
\begin{eqnarray*}
\Phi^\Sh & \equiv & \forallst \tup x\existsst  \tup y \, \phi(\tup x, \tup y, \tup a), \\
(\Phi_\kr)^\D & \equiv & \existsst \tup x \forallst \tup y \, \theta(\tup x,\tup y, \tup a)
\end{eqnarray*}
and $\epastar \vdash \varphi \leftrightarrow \lnot \theta$, as in Lemma \ref{p:ShD}.

Now, suppose that $\Delta_\intern$ is a set of internal formulas and $\Phi(\tup a)$ is a formula provable in $\epaststar$ from  $\Delta_\intern$. We first apply soundness of the Krivine negative translation (Theorem~\ref{l:kukr}) to see that
\[ \ehanststar + \Delta_\intern^\kr \vdash \Phi^\kr, \]
where $\Phi^\kr \equiv \lnot \Phi_\kr$. So if $(\Phi_\kr)^\D \equiv \existsst \tup x \forallst \tup y \, \theta(\tup x,\tup y, \tup a)$, then
\[ (\Phi^\kr)^\D \equiv \existsst \tup Y \forallst \tup x \exists \tup y \in \tup Y[\tup x] \lnot \theta(\tup x, \tup y, \tup a). \]
It follows from the soundness theorem for $\D$ (Theorem \ref{soundnessDst}) and Remark \ref{hanst} that there is a sequence of closed terms $\tup s$ from $\Tstar$ such that
\[ \ehastar + \Delta_\intern^\kr \vdash \forall \tup x \exists \tup y  \in \tup s[\tup x] \lnot \theta(\tup x, \tup y, \tup a). \]
Since $\epastar \vdash \Delta_\intern^\kr \leftrightarrow \Delta_\intern$ and $\epastar \vdash \varphi \leftrightarrow \lnot \theta$ we have
\[
\epastar + \Delta_{\intern} \vdash\ \forall \tup x\exists \tup y \in \tup t(\tup x)\ \phi(\tup x,\tup y, \tup a),
\]
with $\tup t \equiv \lambda \tup x. \tup s[\tup x]$.
\end{proof}

\subsection{Characteristic principles}

The characteristic principles of our functional interpretation for classical arithmetic are idealization $\I$ (or, equivalently, $\R$: see Section 4.1) and $\HAC_\intern$
\[     \forallst x \existsst y \, \varphi(x,y) \to \existsst F \forallst x \exists y \in F(x)\,  \varphi (x,y),         \]
which is the choice scheme $\HAC$ restricted to internal formulas. To see this, note first of all that we have:

\begin{prop}
For any formula $\Phi$ in the language of $\epaststar$ one has:
\[ \epaststar + \I + \HAC_\intern \vdash \Phi \leftrightarrow \Phi^\Sh. \]
\end{prop}
\begin{proof}
An easy proof by induction on the structure of $\Phi$, using $\HAC_\intern$ for the case of negation and $\I$ (or rather $\R$) in the case of internal universal quantification.
\end{proof}

For the purpose of showing that $\I$ and $\HAC_\intern$ are interpreted, it will be convenient to consider the ``hybrid'' system $\ehanststar + \LEM_\intern$, where $\LEM_\intern$ is the law of excluded middle for internal formulas. For this hybrid system we have the following easy lemma, whose proof we omit:

\begin{lemma} We have:
\begin{enumerate}
\item[1.] $\ehanststar + \LEM_\intern \, \vdash \, \varphi^\ku \leftrightarrow \varphi$, if $\varphi$ is an internal formula in the the language of $\epaststar$.
\item[2.] $\ehanststar + \LEM_\intern + \I \, \vdash \, \I^\ku$.
\item[3.] $\ehanststar + \LEM_\intern + \HAC_\intern + \HGMP \, \vdash \, \HAC^\ku_\intern$.
\end{enumerate}
\end{lemma}

This means we can strengthen Theorem  \ref{soundnessshoenfield} to:

\begin{thm} {\rm (Soundness of the $\Sh$-interpretation, full version.)}\label{fullsoundnessSst} Let $\Phi(\tup a)$ be a formula in the language of $\epaststar$ and suppose $\Phi(\tup a)^\Sh\equiv\forallst \tup x \, \existsst \tup y \, \phi(\tup x, \tup y, \tup a)$. If $\Delta_{\intern}$ is a collection of internal formulas and
\[ \epaststar + \I + \HAC_\intern + \Delta_{\intern} \vdash \Phi(\tup a), \]
then one can extract from the formal proof a sequence of closed terms $\tup t$ in $\Tstar$ such that
\[
\epastar + \Delta_{\intern} \vdash\ \forall \tup x\exists \tup y\in \tup t(\tup x)\ \phi(\tup x,\tup y, \tup a).
\]
\end{thm}
\begin{proof} The argument is a slight extension of the proof of Theorem \ref{soundnessshoenfield}. So, once again, let $\Phi(\tup a)$ be a formula in the language of $\epaststar$ and $\varphi$ and $\theta$ be such that
\begin{eqnarray*}
\Phi^\Sh & \equiv & \forallst \tup x\existsst  \tup y \, \phi(\tup x, \tup y, \tup a), \\
(\Phi_\kr)^\D & \equiv & \existsst \tup x \forallst \tup y \, \theta(\tup x,\tup y, \tup a)
\end{eqnarray*}
and $\epastar \vdash \varphi \leftrightarrow \lnot\theta$, as in Lemma \ref{p:ShD}.

This time we suppose $\Delta_\intern$ is a set of internal formulas and $\Phi(\tup a)$ is a formula provable in $\epaststar$ from $\I + \HAC_\intern + \Delta_\intern$. We first apply soundness of the Kuroda negative translation (Theorem~\ref{soundnesskuroda}), which yields:
\[ \ehanststar + \I^\ku + \HAC^\ku_\intern + \Delta_\intern^\ku \vdash \Phi^\ku. \]
Then the previous lemma implies that:
\[ \ehanststar + \LEM_\intern + \I + \HAC_\intern + \HGMP + \Delta_\intern^\ku \vdash \Phi^\ku. \]
Note that $\ehanststar \vdash \Phi^\ku \leftrightarrow \Phi^\kr$,  $\Phi^\kr \equiv \lnot \Phi_\kr$ and
\[ (\Phi^\kr)^\D \equiv \existsst \tup Y \forallst \tup x \exists \tup y \in \tup Y[\tup x] \lnot \theta(\tup x, \tup y , \tup a). \]
Therefore the soundness theorem for $\D$ (Theorem \ref{soundnessDst}), in combination with Remark \ref{hanst} and the fact that the axiom scheme $\LEM_\intern$ is internal, implies that there is a sequence of closed terms $\tup s$ from $\Tstar$ such that
\[ \ehastar + \LEM + \Delta_\intern^\ku \vdash \forall \tup x \exists \tup y \in \tup s[\tup x] \lnot \theta(\tup x, \tup y, \tup a). \]
Since $\epastar \vdash \LEM$, $\epastar \vdash \Delta_\intern^\ku \leftrightarrow \Delta_\intern$ and $\epastar \vdash \varphi \leftrightarrow \lnot \theta$, we have
\[
\epastar + \Delta_{\intern} \vdash\ \forall \tup x\exists \tup y \in \tup t(\tup x)\ \phi(\tup x,\tup y, \tup a)
\]
with $\tup t \equiv \lambda \tup x. \tup s[\tup x]$.
\end{proof}

The following picture depicts the relation between the various interpretations we have established:
\begin{figure}[hbt]%
\[
\begin{xy}
  \xymatrix{
      \epaststar + \I + \HAC_\intern \ar[rd]^{(\cdot)\ku} \ar[dd]_{(\cdot)^\Sh} &  \\
      {} & \ehanststar+\LEM_\intern+\I+\NCR + \HAC + \HGMP + \HIP_{\forallst} \ar[ld]^{(\cdot)^\D}  \\
                                   \epastar &   {}
  }
\end{xy}\]
\caption{The Shoenfield and negative Dialectica interpretations.}
\label{f:ShD}
\end{figure}
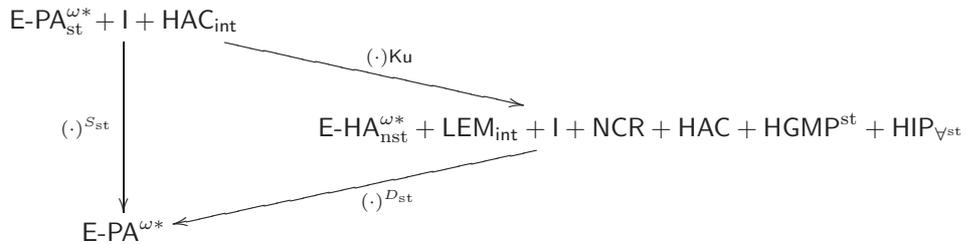

\subsection{Conservation results and the transfer principle}

Theorem \ref{fullsoundnessSst} immediately gives us the following conservation result:

\begin{cor}
$\epaststar + \I + \HAC_\intern$ is a conservative extension of $\epastar$ and hence of $\epa$.
\end{cor}

We conjecture that this result is not the best possible and that $\epaststar + \I + \HAC_\intern + \TPA$ is also conservative over $\epastar$. This would follow from:

\begin{con} \label{consconj}
Let $\Phi(\tup a)$ be a formula in the language of $\epaststar$ and suppose $\Phi(\tup a)^\Sh\equiv\forallst \tup x \, \existsst \tup y \, \phi(\tup x, \tup y, \tup a)$. If $\Delta_{\intern}$ is a collection of internal formulas and
\[ \epaststar + \I + \HAC_\intern + \TPA + \Delta_{\intern} \vdash \Phi(\tup a), \]
then
\[
\epastar + \Delta_{\intern} \vdash\ \forall \tup x\exists \tup y \phi(\tup x, \tup y, \tup a).
\]
\end{con}

Unfortunately, one cannot prove this by showing that one can strengthen the hypothesis of Theorem \ref{fullsoundnessSst} from
$\epaststar + \I + \HAC_\intern + \Delta_{\intern} \vdash \Phi(\tup a)$ to
$\epaststar + \I + \HAC_\intern + \TPA + \Delta_{\intern} \vdash \Phi(\tup a)$, for this strengthened version fails. To see why, note that the $\Sh$-interpretation of $\TPA$
\[ \forallst \tup t \, ( \, \forallst x \,  \varphi(x, \tup t) \to \forall x \, \varphi(x, \tup t) \, ) \]
is provably equivalent to
\[ \forallst \tup t \, \existsst y \, ( \, \varphi(y, \tup t) \to \forall x \, \varphi(x, \tup t) \, ). \]
Therefore such a strengthened version of Theorem \ref{fullsoundnessSst} would imply that for any formula in the language of $\epastar$ without parameters $\varphi(x)$ there are terms $t_1, \ldots, t_n$ such that
\[ \epastar \vdash \bigwedge_i \varphi(t_i) \to \forall x \, \varphi(x). \]
To refute this general statement, it suffices to consider a quantifier-free formula $\varphi(x)$ such that $\forall x \, \varphi(x)$ is true, but  not provable in $\epa$ (such as ``$x$ is not the G\"odel number of a proof in $\epa$ of $\bot$''). This last argument does not refute the conjecture, for the statement
\[ \forall \tup t \, \exists y \, ( \, \varphi(y, \tup t) \to \forall x \, \varphi(x, \tup t) \, ) \]
is a tautology and hence provable in $\epastar$.

Still, we expect that adding $\TPA$ to $\epaststar + \I + \HAC_\intern$ does not destroy conservativity over $\epastar$, because transfer is part of many similar nonstandard systems that have been shown to be conservative over classical base theories (see \cite{palmgren00} and \cite{nelson88}). One natural way to attack this problem would be to try to prove Conjecture \ref{consconj} along the lines of \cite{nelson88}. We plan to take up these issues in future work.

\section{Conclusion and plans for future work}

We hope this paper lays the groundwork for future uses of functional interpretations to analyse nonstandard arguments and systems. There are many directions, both theoretical and applied, in which one could further develop this research topic. We conclude this paper by mentioning a few possibilities which we would like to take up in future research.

First of all, we would like to see if the interpretations that we have developed in this paper could be used to ``unwind'' or ``proof-mine'' nonstandard arguments. Nonstandard arguments have been used in areas where proof-mining techniques have also been successful, such as metric fixed point theory (for methods of nonstandard analysis applied to metric fixed point theory, see \cite{aksoykhamsi90, kirk03}; for application of proof-mining to metric fixed point theory, see \cite{briseid09, gerhardy06, kohlenbach05, kohlenbachleustean03, kohlenbachleustean10, leustean07}) and ergodic theory (for a nonstandard proof of an ergodic theorem, see \cite{kamae82}; for applications of proof-mining to ergodic theory, see \cite{avigad09, avigadgerhardytowsner10, gerhardy08, gerhardy10, kohlenbach11, kohlenbachleustean09, safarik11}), therefore this looks quite promising. For the former type of applications to work in full generality, one would have to extend our functional interpretation to include types for abstract metric spaces, as in \cite{gerhardykohlenbach08, kohlenbach05b}.

But there are also a number of theoretical questions which still need to be answered. Several have been mentioned already: for example, mapping the precise relationships between the nonstandard principles that we have introduced. Another question was whether $\epaststar + \I + \HAC_\intern + \TPA$ is conservative over $\epastar$. Another question is whether our methods allow one to prove conservativity results over $\weha$ and $\wepa$ as well: this will be important if one wishes to combine the results presented here with the proof-mining techniques from \cite{Kohlenbach08}. 

In addition, we would also like to understand the use of saturation principles in nonstandard arguments. These are of particular interest for two reasons: first, they are used in the construction of Loeb measures, which belong to one of the most successful nonstandard techniques. Secondly, for certain systems it has turned out that extending them with saturation principles has resulted in an increase in proof-theoretic strength (see \cite{hensonkeisler86, keisler07}). 

The general saturation principle is
\[ \SAT: \quad \forallst x^\sigma \, \exists y^\tau \, \Phi(x, y) \to \exists f^{\sigma \to \tau} \, \forallst x^\sigma \, \Phi(x, f(x)). \]
Whether this principle has a $\D$-interpretation within G\"odel's $\Tstar$, we do not know; but
\[ \CSAT: \quad \forallst n^0 \, \exists y^\tau \, \Phi(n, y) \to \exists f^{0 \to \tau} \, \forallst n^0 \, \Phi(n, f(n)) \]
has and that seems to be sufficient for the construction of Loeb measures. Interpreting $\CSAT$ and $\SAT$ in the classical context using the $\Sh$-interpretation is probably quite difficult and it is possible that they require some form of bar recursion. We hope to be able to clarify this in future work.

\bibliographystyle{plain} \bibliography{dst}

\end{document}